\documentclass[11pt]{article}
\usepackage{amssymb}
\usepackage{epsfig,amsfonts,amssymb,amsthm}
\usepackage{amsmath}
\usepackage{hyperref}
\newtheorem{theorem}{Theorem}[section]

\newtheorem{lemma}[theorem]{Lemma}
\newtheorem{proposition}{Proposition}

\newtheorem{definition}[theorem]{Definition}
\theoremstyle{definition}
\newtheorem{remark}{Remark}

\numberwithin{equation}{section}

\begin{document}

%\markboth{M.I. Cherdantsev}{Homogenization in high contrast media
%with a defect}

\title{\textbf{Spectral convergence for high contrast elliptic periodic problems with a defect via homogenization}}

\author{M.I. Cherdantsev$\dag$\thanks{The author thanks  V.P. Smyshlyaev and I.V. Kamotski for their help and attention to this work.
The author also thanks V.V. Kamotski for valuable suggestions.
\vspace{6pt}} \\\vspace{6pt}  $\dag$Department of Mathematical
Sciences, University of Bath, Bath, BA2 7AY, UK\\ Email:
mic22@bath.ac.uk}

\maketitle

\maketitle

\begin{abstract}
We consider an eigenvalue problem for a divergence form elliptic
operator $A_\varepsilon$ with high contrast periodic coefficients
with period $\varepsilon$ in each coordinate, where $\varepsilon$ is
a small parameter. The coefficients are perturbed on a bounded
domain of `order one' size. The local perturbation of coefficients
for such operator could result in emergence of localized waves -
eigenfunctions with corresponding eigenvalues lying in the gaps of
the Floquet-Bloch spectrum. We prove that, for the so-called double
porosity type scaling, the eigenfunctions decay exponentially at
infinity, uniformly in $\varepsilon$. Then, using the tools of
two-scale convergence for high contrast homogenization, we prove the
strong two-scale compactness of the eigenfunctions of
$A_\varepsilon$. This implies that the eigenfunctions converge in
the sense of the strong two-scale convergence to the eigenfunctions
of a two-scale limit homogenized operator $A_0$, consequently
establishing `asymptotic one-to-one correspondence' between the
eigenvalues and the eigenfunctions of these two operators. We also
prove by direct means the stability of the essential spectrum of the
homogenized operator with respect to the local perturbation of its
coefficients. That allows us to establish not only the strong
two-scale resolvent convergence of $A_\varepsilon$ to $A_0$ but also
the Hausdorff convergence of the spectra of $A_\varepsilon$ to the
spectrum of $A_0$, preserving the multiplicity of the isolated
eigenvalues.

\medskip


\medskip


\end{abstract}

\section{Introduction}
In this paper we consider a high contrast two-phase periodic medium
with a small period and with a `finite size' defect filled by a
third phase, see Fig. \ref{fig1}. This physically represents, for
instance, a simplified model of cross-section of a photonic crystal
fiber. Mathematically, the problem relates to a compact perturbation
of $\varepsilon$-periodic coefficients in a divergence form elliptic
operator $A_\varepsilon$. The behaviour of $A_\varepsilon $ and its
spectral characteristics as $\varepsilon \to 0$ is of the main
interest. A similar problem is considered in \cite{SmysKam} using
the method of asymptotic expansions, but the present study pursues
different aims and approaches the problem from another direction,
namely developing an appropriate version of the two-scale
convergence technique \cite{Ngu, Al, Zhikov2}. As a result we obtain
a complete description of the asymptotic (with respect to
$\varepsilon$) behaviour of the localized modes and other spectral
characteristics for the operator $A_\varepsilon$ in terms of an
explicitly described (two-scale) limit operator $A_0$. For other
recent applications of the high contrast homogenization techniques
see also \cite{Briane, Bouchitte, Bellieud, CSZH, CH, Babych}.

In the absence of a defect, Zhikov considers in \cite{Zhikov1} a
divergence form elliptic operator $\widehat{A}_\varepsilon$ (denoted
by $A_\varepsilon$ in \cite{Zhikov1}) with periodic coefficients
corresponding to a double-porosity model \cite{ADH, BMP}
($A_\varepsilon$ in our notation is obtained from
$\widehat{A}_\varepsilon$ by a compact perturbation of its
coefficients). Operators of such type have the Floquet-Bloch
essential spectrum, displaying a band-gap structure. Zhikov proves
that the spectra of $\widehat{A}_\varepsilon$ converge in the sense
of Hausdorff to the spectrum of a certain two-scale homogenized
operator $\widehat{A}_0$ with constant coefficients, see also
\cite{Hempell, Zhikov2}, and that $\widehat{A}_0$ is the limit of
$\widehat{A}_\varepsilon$ in the sense of strong two-scale resolvent
convergence. The spectrum of $\widehat{A}_0$ is purely essential and
displays an explicit band-gap structure. It is well known, see e.g.
\cite{ReedSimon, FigKlein}, that in the case of a compact
perturbation of periodic coefficients in the elliptic operator
$\widehat{A}_\varepsilon$ its essential spectrum remains
unperturbed. The only extra spectrum that can emerge in the gaps due
to the perturbation is a discrete one (isolated eigenvalues with
finite multiplicity).\footnote[1]{We do not concern in this paper
the issue of whether \textit{embedded} eigenvalues can emerge on the
bands as a result of the perturbation.} Such an extra spectrum does
emerge at least under some assumptions, e.g. \cite{FigKlein,
SmysKam}. The latter corresponds physically to localized modes
emerging near the defect. In order to establish the strong two-scale
convergence of the eigenfunctions of $A_\varepsilon $ we need their
strong two-scale compactness. The latter requires in turn an
exponential decay of the eigenfunctions \textit{uniform in
$\varepsilon$}.

The problem of wave localization (i.e. of the existence of
eigenvalues with corresponding eigenfunctions decaying
exponentially) in the gaps of the essential spectrum has been
intensively investigated for a wide range of differential operators
over the last decades. The results obtained up to date ensure the
exponential decay of eigenfunctions of $A_\varepsilon$ for a
\textit{fixed} $\varepsilon $, see e.g. \cite{FigKlein}. However
this is insufficient for establishing the required compactness.
Moreover, the developed methods, e.g. \cite{BCH} and \cite{FigKlein}
(the latter using the method of Agmon\cite{Agmon}), seem to be
insufficient for the present purpose. The reason is that in order to
obtain the \textit{uniform} exponential decay one has to perform
some kind of two-scale asymptotic analysis, investigating the
behaviour of the eigenfunctions on small and large scales
simultaneously. To achieve this we supplement the method of
\cite{Agmon} by the related two-scale techniques, which play a
crucial role. As a result, we obtain a uniform estimate with the
decay exponent $\alpha$ (see (\ref{15a}) and (\ref{11}) below) which
ensures the compactness, but may also be of an independent interest.
On one hand, it is sharp in a sense. On the other hand, it behaves
qualitatively entirely different compared to e.g. the one in
\cite{BCH}: while the one in \cite{BCH} is proportional to the
square root of the distance to the gap end, the decay exponent we
derive becomes large on approaching the left end of the gap and
small near the right end.

The structure of the paper is the following. We first define the
problem in Section \ref{s2}, describe the two-scale limit operator
$A_0$ and state the main result. We then consider a subsequence of
eigenvalues of $A_\varepsilon$ converging to some point $\lambda_0$
lying in a gap of the spectrum of $\widehat{A}_0$. In Section
\ref{s3} we prove (Theorem \ref{theorem1}) the uniform exponential
decay for the eigenfunctions of $A_\varepsilon$. Section \ref{s4} is
devoted to the proof of a main auxiliary lemma that is employed in
the previous section, which may also be of an independent interest.
In Section \ref{s5} we list some properties of the two-scale
convergence and several related statements which we use in the next
section. Employing the uniform exponential decay, we establish in
Section \ref{s6} (see Theorem \ref{th6}) the strong two-scale
compactness of (normalized) eigenfunctions of $A_\varepsilon$, see
e.g. \cite{Zhikov2, Zhikov1}. This implies that, up to a
subsequence, the eigenfunctions two-scale converge to a function,
which is eventually proved to be an eigenfunction of the two-scale
limit operator $A_0$ with a defect, which could be considered as a
perturbation of $\widehat{A}_0$. Accordingly $\lambda_0$ is an
eigenvalue of $A_0$. The two-scale convergence of the eigenfunctions
together with the results of \cite{SmysKam} on the existence of the
eigenvalues in the gaps and related error bounds allow us to make a
conclusion about the `asymptotic one-to-one correspondence' between
eigenfunctions and eigenvalues of the operators $A_\varepsilon$ and
$A_0$ as $\varepsilon \to 0$. In the last section we prove by direct
means (via the Weyl's sequences) the stability of the essential
spectrum of $\widehat{A}_0$ with respect to the local perturbation
of its coefficients (see Theorem \ref{th7}). Thereby this
establishes the convergence of the spectra of $A_\varepsilon$ to the
spectrum of $A_0$ in the sense of Hausdorff (Theorem \ref{th1}).

\section{Notation, problem formulation, limit operator and the main result}\label{s2}
We will use the following notation for the geometric configuration
visualized on Figure \ref{fig1}, cf. \cite{SmysKam}. Consider a
periodic set of unit cubes
\begin{equation}\label{1}
\{Q: Q=[0,1)^n + \xi, \, \xi \in \mathbb{Z}^n\}.
\end{equation}
Let $F_0$ be an open periodic set with period one in each coordinate
such that $F_0 \cap Q \Subset Q$ is a connected domain with
infinitely smooth boundary. We denote $F_0 \cap Q$ by $Q_0$ and
$Q\backslash\overline Q_0$ by $Q_1$. Notice that the position of the
particular set $Q_0, \, Q_1$ or $Q$ depends on $\xi\in
\mathbb{Z}^n$, however we will not reflect this fact in the notation
to simplify the latter. Regularity assumptions on the boundary could
be relaxed.\footnote[2]{In particular, the results on the two-scale
convergence stated in the paper remain valid at least under the
assumption of Lipschitz regular boundaries. The
$\varepsilon^{1/2}$-order bounds, as they were obtained in
\cite{SmysKam}, require higher regularity.} Let $\Omega_2$ be a
bounded domain with a sufficiently smooth boundary, containing the
origin; its complement is denoted by $\Omega_1$, $\Omega_1 =
\mathbb{R}^n \backslash \overline{\Omega}_2$.

\begin{figure}
\centerline{\epsfbox{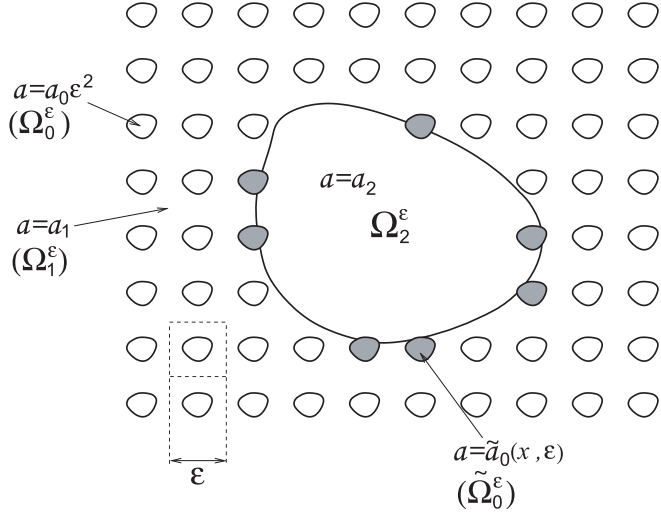}} \caption{A defect in a rapidly
oscillating high contrast periodic medium, cf. \cite[Fig.
1]{SmysKam}.} \label{fig1}
\end{figure}

We define the `inclusion phase' or the `soft phase'
$\Omega^\varepsilon_0$ as
\[
\Omega^\varepsilon_0 = \bigcup_{\varepsilon Q_0
\subset\Omega_1} \varepsilon Q_0, %\label{eq:eq1}
\]
where $\varepsilon > 0$ is a small parameter. The set of inclusions
$\varepsilon Q_0$ which intersect the boundary of $\Omega_2$ is
denoted by $\widetilde{\Omega}^\varepsilon_0$. The `matrix phase',
denoted by $\Omega^\varepsilon_1$, is the complement to the
inclusions in $\Omega_1$, i.e. $\Omega^\varepsilon_1 =
\Omega_1\backslash
\overline{(\Omega^\varepsilon_0\cup\widetilde{\Omega}^\varepsilon_0)}$.
`Defect domain' $\Omega^\varepsilon_2$ is defined by $\Omega_2
\backslash \overline{\widetilde{\Omega}^\varepsilon_0}$. We also use
the notation $\theta_\Omega$ for the characteristic function of a
set $\Omega$ and $B_R$ for the open ball of radius $R$ centered at
the origin.

We consider an elliptic operator $A_\varepsilon$, self-adjoint in
$L^2\left(\mathbb{R}^n\right)$,
\begin{equation}
   A_\varepsilon u^\varepsilon\,:=\, -\nabla\cdot \biggl(a(x,\varepsilon) \nabla
u^\varepsilon(x)\biggr),\ \  x\in  \mathbb{R}^n,
      \label{3}
\end{equation}
in particular the eigenvalue problem
\begin{equation}
   A_\varepsilon u^\varepsilon\,=\,\lambda_\varepsilon
    u^\varepsilon
      \label{2}
\end{equation}
for its point spectrum. The coefficient $a(x,\varepsilon)$ is given
by the formula
\begin{equation}
a(x,\varepsilon)=
    \left\{
      \begin{array}{ll}
         a_0\varepsilon^2, & x\in \Omega_0^\varepsilon,
          \\
          a_1, & x\in \Omega_1^\varepsilon,
          \\
         a_2, & x\in \Omega_2^\varepsilon,
          \\
         \widetilde a_0(x,\varepsilon), & x\in \widetilde\Omega_0^\varepsilon,
      \end{array} \right.
      \label{4}
\end{equation}
where measurable $\widetilde{a}_0(x,\varepsilon)$ is such that
\begin{equation}\label{5a}
\mbox{either } \widetilde{A}_0\, \varepsilon^{2-\theta} \leq
\widetilde{a}_0(x,\varepsilon) \leq \widetilde{B}_0 \,
\varepsilon^{2-\theta} \mbox{ for all } \varepsilon, \mbox{ or }
\widetilde{a}_0(x,\varepsilon) = a_0\, \varepsilon^2 \mbox{ for all
} \varepsilon.
\end{equation}
Here $a_0$, $a_1$, $a_2$, $\widetilde{A}_0$, $\widetilde{B}_0$ and
$\theta$ are some positive constants independent of $\varepsilon$,
$\theta \in (0,2]$. Notice that this includes as particular cases
e.g. the case of `removed' boundary inclusions, i.e.
$a(x,\varepsilon)=a_1$ if $x\in \widetilde\Omega_0^\varepsilon \cap
\Omega_1$, $a(x,\varepsilon)=a_2$ if $x\in
\widetilde\Omega_0^\varepsilon \cap \Omega_2$, and the case of the
`full' inclusions, $\widetilde{a}_0(x,\varepsilon) = a_0\,
\varepsilon^2$. The domain of $A_\varepsilon$ is defined in a
standard way via Friedrichs extension procedure with a bilinear
form, see (\ref{5}) below, defined on $H^1(\mathbb R^n)$.

For any $\varepsilon > 0$ the operator $A_\varepsilon$ is an
operator with $\varepsilon$-periodic coefficients, which are
compactly perturbed (within bounded domain $\Omega_2^\varepsilon
\cup \widetilde\Omega_0^\varepsilon$). This implies (e.g.
\cite{ReedSimon, FigKlein}) that its essential spectrum coincides
with the Floquet-Bloch spectrum of the associated `unperturbed'
operator $\widehat{A}_\varepsilon$, with only extra spectrum being
hence the discrete spectrum in the gaps of
$\widehat{A}_\varepsilon$.\footnote[3]{This does not rule out
possible emergence of embedded eigenvalues on the bands, not
considered in this paper.} Note that the spectrum of
$\widehat{A}_\varepsilon$ contains gaps for small enough
$\varepsilon$, cf. \cite{Hempell, Zhikov2, Zhikov1}, and there is
often an extra discrete spectrum in the gaps of $\sigma_{\textrm
ess}( A_\varepsilon)$, e.g. \cite{SmysKam}. By definition,
$u^\varepsilon \in H^1(\mathbb{R}^n)$, $u^\varepsilon \not\equiv 0$,
is an eigenfunction of the eigenvalue problem (\ref{2}) with an
eigenvalue $\lambda_\varepsilon$ if
\begin{equation}
\int\limits_{\mathbb{R}^n} a(x,\varepsilon) \nabla u^\varepsilon
\cdot \nabla w \,dx = \lambda_\varepsilon \int\limits_{\mathbb{R}^n}
u^\varepsilon w \,dx \label{5}
\end{equation}
for all $w \in H^1(\mathbb{R}^n)$.

The aim of this work is to establish that as $\varepsilon \to 0$ the
operator $A_\varepsilon$ converges in the appropriate sense (namely,
in the sense of two-scale convergence, see Section \ref{s5}) to a
`two-scale' limit operator $A_0$, which we describe next. For the
rest of the present section we assume that $Q=[0,1)^n$, considering
all functions of two variables $(x,y)$ to be 1-periodic in each
coordinate with respect to $y$. The `two-scale' limit operator $A_0$
is analogous to the one introduced in the defect free setting by
Zhikov \cite{Zhikov2, Zhikov1} and acts in a Hilbert space
\begin{equation}
\mathcal {H}_0:=\biggl\{ u(x,y)\in L^2\left(\mathbb{R}^n\times
Q\right) \biggl\vert \, u(x,y)=u_0(x)+v(x,y), u_0\in
L^2\left(\mathbb{R}^n\right), \biggl. v\in\,L^2\left(\Omega_1;
\,L^2(Q_0)\right) \biggr\},  \label{5.1}
\end{equation}
with the natural inner product inherited from $L^2(\mathbb{R}^n
\times Q)$ and $\mathcal {H}_0$ being its closed subspace, cf.
\cite{Zhikov1}. It is implied that $v$ is extended by zero for $y\in
Q_1$ or $x \in \Omega_2$. The operator $A_0$ is defined as generated
by a (closed) symmetric and bounded from below bilinear form
$B_0(u,w)$ acting in a dense subspace
\begin{equation}
\mathcal {V}\, =
\,H^1\left(\mathbb{R}^n\right)\,+\,L^2\left(\Omega_1,
H^1_0(Q_0)\right)
 \label{Qdomain}
\end{equation}
of $\mathcal {H}_0 =
L^2\left(\mathbb{R}^n\right)\,+\,L^2\left(\Omega_1,
L^2(Q_0)\right)$, which is defined as follows: for $u=u_0+v, w=w_0+z
\in\,\mathcal {V}$,
\begin{equation}
 B_0(u,w)\,=\,a_2\int\limits_{\Omega_2}\nabla u_0 \cdot  {\nabla w_0} \,dx+
\int\limits_{\Omega_1}A^{\rm hom}\nabla u_0 \cdot  {\nabla w_0} \,dx
+ a_0 \int\limits_{\Omega_1} \int\limits_{Q_0} \nabla_y v \cdot
{\nabla_y z} \,dy \,dx. \label{6}
\end{equation}
Here $A^{\rm hom}=\left(A^{\rm hom}_{ij}\right)$ is the standard
``porous'' homogenized (symmetric, positive-definite) matrix for the
periodic medium as described above but when no defect is present and
with $a_0=0$, see e.g. \cite[\S3.1]{ZhiKozOle}:
\begin{equation}
A^{\rm hom}_{ij}\xi_i\xi_j=\inf_{w\in C^\infty_{\rm per}(Q)}
\int\limits_{Q_1}a_1\vert\xi+\nabla w\vert^2\,dy\,\,\,\,\left(\xi
\in \mathbb R^n\right). \label{eq:ahom}
\end{equation}
Here $C^\infty_{\textrm per}(Q)$ stands for the set of infinitely
smooth functions with periodic boundary conditions. Then one can see
(cf. \cite{Zhikov1}) that the form is indeed bounded from below,
densely defined and closed. Hence, according to the standard
Friedrichs extension procedure, e.g. \cite{ReedSimon}, $A_0$ can be
defined as a self-adjoint operator with a domain $\mathcal D (A_0)
\subset \mathcal {V}$. A function $u^0(x,y) = u_0(x) + v(x,y) \in
\mathcal{V}$, $u^0(x,y) \not\equiv 0$, is an eigenfunction of the
limit operator $A_0$ corresponding to an eigenvalue $\lambda_0$ if
and only if
\begin{equation}
B_0(u^0,w) = \lambda_0 \int\limits_{\mathbb{R}^n} \int\limits_{Q}
(u_0 + v) (w_0 + z) \,dy \,dx . \label{10}
\end{equation}
for any $w=w_0+z \in\,\mathcal{V}$  (we assume where it is possible
that a function defined on a smaller domain is extended by zero on a
larger domain).\footnote[4]{Explicit example in \cite[\S 5]{SmysKam}
ensure the existence of isolated eigenvalues of $A_0$ of finite
multiplicity in the gaps of $\sigma(\widehat{A}_0)$ in a particular
situation.}

The `unperturbed' operators $\widehat{A}_\varepsilon$ and
$\widehat{A}_0$ could be defined analogously to $A_\varepsilon$ and
$A_0$ formally setting above $\Omega_2 = \emptyset$ and $\Omega_1 =
\mathbb{R}^n$. (See also \cite{Zhikov2, Zhikov1}, where these
operators are denoted by $A_\varepsilon$ and $A$ respectively.)

\begin{figure}
\centerline{\epsfbox{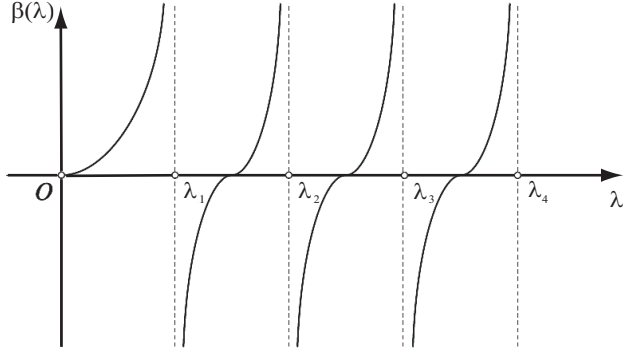}} \caption{$\beta(\lambda )$, cf.
\cite{Zhikov1}.} \label{fig2}
\end{figure}

We next describe a function $\beta(\lambda)$ which was introduced by
Zhikov \cite{Zhikov2, Zhikov1} (cf. also \cite{Bouchitte}) and plays
an important role in our considerations. Let $\lambda_j$ and
$\varphi_j$, $j=1,2,\ldots$, be eigenvalues and corresponding
orthonormalized eigenfunctions of operator $T$ defined as
\begin{equation}
T f:= - a_0 \Delta f, \quad f\in H^1_0(Q_0) \cap H^2(Q_0). \label{8}
\end{equation}
Note that the eigenvalues of $T$ belong to the spectrum of
$\widehat{A}_0$, see \cite{Zhikov2}. For $\lambda \neq \lambda_j$,
$j\geq 1$, denote by $b$ the solution to
\begin{equation}
T b - \lambda b = - a_0 \Delta b - \lambda b =   1, \,\, b\in
H^1_0(Q_0). \label{13a}
\end{equation}
The function $\beta(\lambda)$ is defined by
\begin{equation}
\beta(\lambda):= \lambda \big(1 + \lambda \langle b \rangle_y \big)
= \,\lambda\,+\,\lambda^2 \sum_{j=1}^\infty
\frac{\langle\varphi_j\rangle_y^2}{\lambda_j-\lambda} \label{11}
\end{equation}
where $\langle f \rangle_y : = \int\limits_{Q} f(y) \,dy$. It is
well-defined for any $\lambda $ except $\lambda = \lambda_j$ with
$\langle\varphi_j\rangle_y \neq 0$, monotonically increasing between
such points, see Figure \ref{fig2}. This function describes the
structure of $\sigma(\widehat{A}_0)$, see \cite{Zhikov2}. Namely,
the intervals where $\beta(\lambda) \geq 0$ correspond to the bands
of the spectrum of $\widehat{A}_0$. Isolated points of the spectrum
of $\widehat{A}_0$, i.e. $\lambda_j$ such that $\langle \varphi_j
\rangle_y=0$ and $\beta(\lambda_j) < 0$, can also be regarded as
degenerate bands. The intervals on which $\beta(\lambda) < 0$
(excluding $\lambda_j$) are gaps.

It was shown in \cite{Zhikov1} (see also \cite{Hempell, Zhikov2})
that $\sigma(\widehat{A}_\varepsilon)$ converges in the sense of
Hausdorff to $\sigma(\widehat{A}_0)$, while
$\widehat{A}_\varepsilon$ converges to $\widehat{A}_0$ in the sense
of the strong two-scale resolvent convergence (cf. Sections \ref{s5}
and \ref{s6} below) implying the convergence of spectral projectors,
etc.

We aim at showing that similar as well as some further results hold
for the perturbed operators. Namely, our main result is the
following
\begin{theorem}\label{th1}
The operator $A_\varepsilon$ converges to $A_0$ in the sense of the
strong two-scale resolvent convergence. Hence the spectral
projectors also strongly two-scale converge away from the point
spectrum of $A_0$. The spectrum of $A_\varepsilon$ converges in the
sense of Hausdorff to the spectrum of $A_0$. Let $\lambda_0$ be an
isolated eigenvalue of multiplicity $m$ of the operator $A_0$ in the
gap of its essential spectrum. Then, for small enough $\varepsilon$,
there exist exactly $m$ eigenvalues $\lambda_{\varepsilon,i}$ of
$A_\varepsilon$ (counted with their multiplicities) such that
\begin{equation}
|\lambda_{\varepsilon,i} - \lambda_0| \leq C \varepsilon^{1/2}, \,
i=1,\ldots,m, \label{12a}
\end{equation}
with a constant $C$ independent of $\varepsilon $.\footnote[5]{The
error bound (\ref{12a}) employs the results of \cite{SmysKam}
requiring, as stated, higher regularity of $\partial Q_0$. The rest
of the statement of the theorem applies potentially to less regular
boundaries.} If for some sequence $\varepsilon_k\to 0$ a sequence of
eigenvalues $\lambda_{\varepsilon_k}$ of $A_\varepsilon$ converges
to $\lambda_0$ which is in the gap of the essential spectrum of
$A_0$, then, $\lambda_0$ is an isolated eigenvalue of $A_0$ of a
finite multiplicity $m$ and for large enough $k$,
$\lambda_{\varepsilon_k} \in \{\lambda_{\varepsilon_k,i}, \,
i=1,\ldots,m\}$.
\end{theorem}

A key part in establishing the latter is in controlling the
behaviour at infinity of the eigenfunctions corresponding to the
extra point spectrum which may appear in the spectral gaps of the
unperturbed operator. A central property providing this is a uniform
exponential decay of the eigenfunctions which we prove next.

\section{Uniform exponential decay of the eigenfunctions of
$A_\varepsilon$}\label{s3} Let $\lambda _0$ be a point in a gap of
$\sigma(\widehat{A}_0)$, i.e. such that $\beta (\lambda_0)<0$ and
$\lambda_0 \neq \lambda_j$ for all $j$. Assume $\lambda _0$ is an
accumulation point of the point spectra of $A_\varepsilon$, i.e. for
some subsequence $\varepsilon_k \to 0$ there exist eigenvalues
$\lambda_{\varepsilon_k}$ of $A_\varepsilon$ such that
$\lambda_{\varepsilon_k}\to \lambda_0$ as $k\to \infty$. (Notice
that the results of \cite{FigKlein, SmysKam} ensure in particular
that such series do exist.) We formulate the main result of this
section (and also one of the principal results of the paper) in the
following statement.
\begin{theorem}\label{theorem1}
Let $\lambda_{\varepsilon_k}$ and $u^{\varepsilon_k}$ be sequences
of eigenvalues of the operator $A_\varepsilon$ and corresponding
eigenfunctions normalized in $L^2(\mathbb{R}^n)$, where
$\varepsilon_k$ is some positive sequence converging to zero as $k
\to \infty$. Let $\lambda_0$ be such that $\beta(\lambda_0)$ is
negative and $\lambda_0$ is not an eigenvalue of the operator $T$
given by (\ref{8}). Suppose that $\lambda_{\varepsilon_k}$ converges
to $\lambda_0$. Then for small enough $\varepsilon_k$ eigenfunctions
$u^{\varepsilon_k}$ decay uniformly exponentially at infinity,
namely, for
\begin{equation}\label{15a}
0 < \alpha < \sqrt{-\beta(\lambda_0)/a_1}
\end{equation}
the following holds:
\[
\| e^{\alpha|x|} u^{\varepsilon_k} \|_{L^2(\mathbb{R}^n)} \leq C,
\]
uniformly in $\varepsilon_k$, i.e. for any $0 < \varepsilon_k <
\varepsilon(\alpha)$, with $C=C(\alpha)$ independent of
$\varepsilon$.
\end{theorem}
\begin{proof}We drop the index $k$ in $\varepsilon_k$ for
the sake of simplification of notation. So, when we say, for
instance, `sequence $\lambda_\varepsilon$' we actually mean
`subsequence $\lambda_{\varepsilon_k}$'.

The plan of the proof is the following. We first derive `elementary'
a priori estimates for the eigenfunction $u^\varepsilon$ outside the
set of inclusions $\Omega_0^\varepsilon \cup
\widetilde{\Omega}_0^\varepsilon$. Next we study the structure of
the eigenfunction at the small scale and deduce some vital
inequalities for $\varepsilon \nabla u^\varepsilon$ inside the
inclusions. As a central technical step, we then employ in the
integral identity (\ref{5}) a test function with exponentially
growing weight $g^2(|x|)$, see (\ref{4.8})--(\ref{4.9}) below, and
perform some delicate uniform estimates to achieve the result. The
main auxiliary technical results are proven in Lemma \ref{lemma} and
Proposition \ref{prop}.

\textbf{Step 1.} Setting $w=u^\varepsilon$ in (\ref{5}) we have
\[
\varepsilon^2 a_0 \| \nabla u^\varepsilon
\|^2_{L^2(\Omega_0^\varepsilon)} + a_1 \| \nabla u^\varepsilon
\|^2_{L^2(\Omega_1^\varepsilon)} + a_2 \| \nabla u^\varepsilon
\|^2_{L^2(\Omega_2^\varepsilon)} +  \|
\widetilde{a}_0^{1/2}(x,\varepsilon) \, \nabla u^\varepsilon
\|^2_{L^2(\widetilde{\Omega}_0^\varepsilon)} = \lambda_\varepsilon
\|  u^\varepsilon \|^2_{L^2(\mathbb{R}^n)} = \lambda_\varepsilon.
\]
Therefore
\begin{equation}
\| u^\varepsilon
\|_{H^1(\mathbb{R}^n\backslash(\Omega_0^\varepsilon\bigcup\widetilde{\Omega}_0^\varepsilon))}
\leq C \label{12}
\end{equation}
uniformly in $\varepsilon $. From now on $C$ denotes a generic
constant whose precise value is insignificant and can change from
line to line.

\textbf{Step 2.} Let us consider the function $u^\varepsilon$ in a
cell $\varepsilon Q$ corresponding to such $\xi = \xi(\varepsilon)
\in \mathbb Z^n$, see \ref{1}, that the corresponding `inclusion'
$\varepsilon Q_0$ has a nonempty intersection with $\Omega_1$. There
exists an extension $\widetilde{u}^\varepsilon$ of
$u^\varepsilon|_{\varepsilon Q_1}$ to the whole cell $\varepsilon Q$
such that
\begin{equation}
\|  \widetilde{u}^\varepsilon \|_{L^2(\varepsilon Q_0)} \leq C \|
 u^\varepsilon \|_{L^2(\varepsilon Q_1)}, \quad \| \nabla
\widetilde{u}^\varepsilon \|_{L^2(\varepsilon Q_0)} \leq C \| \nabla
u^\varepsilon \|_{L^2(\varepsilon Q_1)}, \label{4.1}
\end{equation}
where $C$ does not depend on $\varepsilon$ or $\xi$, see e. g.
\cite[Ch. 3, \S 4, Th. 1]{Mikh}, which is a version of the so-called
`extension lemma', see also e.g. \cite[\S 3.1, L. 3.2]{ZhiKozOle}.
In particular, we can choose the following extension:
\[
\begin{aligned}
&\widetilde{u}^\varepsilon \equiv u^\varepsilon,& \,\, x\in
\Omega_1^\varepsilon \cup \Omega_2^\varepsilon,
\\
-\nabla\cdot \bigl(a(x,&\varepsilon) \nabla
\widetilde{u}^\varepsilon(x)\bigr)
 = 0,& \,\, x\in \Omega_0^\varepsilon \cup
\widetilde{\Omega}_0^\varepsilon,
\end{aligned}
\]
which minimizes $\|a^{1/2}(x,\varepsilon) \nabla
\widetilde{u}^\varepsilon \|_{L^2(\varepsilon Q_0)}$ subject to the
prescribed boundary conditions, with (\ref{4}) and (\ref{5a})
ensuring that (\ref{4.1}) still holds. From (\ref{12}) and
(\ref{4.1}) we conclude that
\begin{equation}
\| \widetilde{u}^\varepsilon \|_{H^1(\mathbb{R}^n)} \leq C.
\label{130}
\end{equation}
We represent $u^\varepsilon$ in the form
\begin{equation}
u^\varepsilon(x) = \widetilde{u}^\varepsilon(x) + v^\varepsilon(x) %\label{1.7}
\label{18a}
\end{equation}
and consider the function $v^\varepsilon \in
H^1_0(\Omega_0^\varepsilon \cup
\widetilde{\Omega}_0^\varepsilon)$.\footnote[6]{In a sense,
(\ref{18a}) decomposes $u^\varepsilon $ into a slowly varying part
$\widetilde{u}^\varepsilon $ and rapidly varying $v^\varepsilon $.
The two are coupled and subsequently analyzed simultaneously, which
is the essence of two-scale asymptotic analysis.} In each inclusion
$\varepsilon Q_0 \subset \Omega_0^\varepsilon \cup
\widetilde{\Omega}_0^\varepsilon$ we have the following boundary
value problem for $v^\varepsilon(x)$:
\begin{equation}
\begin{aligned}
-\nabla \cdot (a(x,\varepsilon) \nabla v^\varepsilon) -
\lambda_\varepsilon v^\varepsilon =  \lambda_\varepsilon
\widetilde{u}^\varepsilon, \,\, x\in \varepsilon Q_0; \quad
v^\varepsilon(x) =& 0, \,\, x\in
\partial (\varepsilon Q_0).\label{13}
\end{aligned}
\end{equation}
When $a(x,\varepsilon) = a_0 \varepsilon^2$, i.e. everywhere in
$\Omega_0^\varepsilon$ and also in
$\widetilde{\Omega}_0^\varepsilon$ in the case
$\widetilde{a}_0(x,\varepsilon) = a_0 \varepsilon^2$, after changing
the variables $x\to y = x/\varepsilon$ we obtain
\begin{equation}
- a_0 \Delta_y v^\varepsilon(\varepsilon y) - \lambda_\varepsilon
v^\varepsilon(\varepsilon y) = \lambda_\varepsilon
\widetilde{u}^\varepsilon(\varepsilon y), \,\, y\in Q_0, \quad
v^\varepsilon(\varepsilon y) = 0, \,\, y\in
\partial Q_0.\label{4.2}
\end{equation}
Since $\lambda_0 \neq \lambda_j$ by the assumptions of the theorem,
$\lambda_\varepsilon$ is separated uniformly from the spectrum of
operator (\ref{8}) for small enough $\varepsilon $. Hence the
resolvent at $\lambda_\varepsilon$ is bounded uniformly in
$\varepsilon$ and (\ref{4.2}) implies
\begin{equation}
\|v^\varepsilon(\varepsilon y)\|_{H^1(Q_0)} \leq C \|
\widetilde{u}^\varepsilon(\varepsilon y)\|_{L^2(Q_0)}. \label{16}
\end{equation}

In the case when $\widetilde{A}_0\, \varepsilon^{2-\theta} \leq
\widetilde{a}_0(x,\varepsilon) \leq \widetilde{B}_0 \,
\varepsilon^{2-\theta}$, $\theta \in (0,2]$, we multiply equation
(\ref{13}) by $v^\varepsilon$ and integrate by parts to obtain after
rescaling
\begin{equation}
\varepsilon^{-2} \int\limits_{\varepsilon Q_0}
\widetilde{a}_0(\varepsilon y,\varepsilon) |\nabla_y
v^\varepsilon(\varepsilon y)|^2 dx - \lambda_\varepsilon
\int\limits_{\varepsilon Q_0} \big(v^\varepsilon(\varepsilon
y)\big)^2 \, dx = \lambda_\varepsilon \int\limits_{\varepsilon Q_0}
\widetilde{u}^\varepsilon(\varepsilon y)v^\varepsilon(\varepsilon
y)\, dx. \label{15}
\end{equation}
Notice that $\varepsilon^{-2} \widetilde{a}_0 (\varepsilon
y,\varepsilon) \geq \widetilde{A}_0 \varepsilon^{-\theta} \to
\infty$ as $\varepsilon \to 0$. Then using Poincar\'{e} inequality
one easily derives
\begin{equation}
\varepsilon^{-2}\|\widetilde{a}_0^{1/2} \nabla_y
v^\varepsilon(\varepsilon y)\|_{L^2(Q_0)}^2 +
\|v^\varepsilon(\varepsilon y)\|_{L^2(Q_0)}^2 \leq C \|
\widetilde{u}^\varepsilon(\varepsilon y)\|_{L^2(Q_0)}^2, \label{18}
\end{equation}
for small enough $\varepsilon$. Returning in (\ref{16}) and
(\ref{18}) to the variable $x$ we arrive at the following inequality
that describes the behaviour of $v^\varepsilon$ and its gradient in
$\Omega_0^\varepsilon \cup \widetilde{\Omega}_0^\varepsilon$,
\begin{equation}
\|a^{1/2} \nabla v^\varepsilon(x)\|^2 _{L^2(\varepsilon Q_0)} + \|
v^\varepsilon(x)\|^2 _{L^2(\varepsilon Q_0)} \leq C
\|\widetilde{u}^\varepsilon(x)\|^2_{L^2(\varepsilon Q_0)},
\label{4.6}
\end{equation}
with an $\varepsilon$-independent constant $C$.

\textbf{Step 3.} In order to get the uniform exponential decay of
the eigenfunctions we next substitute in (\ref{5}) a test function
of a special form:
\begin{equation}
w = g^2(|x|) \widetilde{u}^\varepsilon(x). \label{4.8}
\end{equation}
Here we define function $g$ as follows
\begin{equation}
g(t) = \left\{
\begin{array}{ll}
e^{\alpha t},\,\, & t\in [0,R],
\\
e^{\alpha R}, \,\, & t\in (R,+\infty),
\end{array}
\right.\label{4.9}
\end{equation}
where $R$ is some arbitrary positive number. The exponent $\alpha$
will be chosen later. This method was employed e.g. by Agmon, see
\cite{Agmon}, but in the present case its realization is not
straightforward. Namely, to obtain the desired estimates we have to
implement the approach of \cite{Agmon} in the context of the
two-scale analysis. We will show that $g(|x|)
\widetilde{u}^\varepsilon(x)$, and consequently $g(|x|)
u^\varepsilon(x)$, are bounded in $L^2(\mathbb{R}^n)$ uniformly with
respect to $R$ and $\varepsilon$. Then we will show via passing to
the limit as $R\to \infty$ that we can replace $g(|x|)$ by
$e^{\alpha |x|}$.
\begin{remark}
We cannot use $e^{2 \alpha |x|} \widetilde{u}^\varepsilon(x)$ as a
test function directly, since it is not known at this stage that
this function is square integrable.
\end{remark}
The following identity holds by direct inspection
\begin{equation}
\nabla \widetilde{u}^\varepsilon \nabla (g^2
\widetilde{u}^\varepsilon) = |\nabla(g\widetilde{u}^\varepsilon)|^2
- |\nabla g|^2 (\widetilde{u}^\varepsilon)^2. \label{24}
\end{equation}
Notice that the absolute value of $\nabla g$ is bounded by $g$ with
 $\alpha$ (uniformly in $R$):
\begin{equation}
\big| \nabla g(|x|) \big| \leq \alpha g(|x|). \label{4.11}
\end{equation}
After the substitution of (\ref{4.8}) into (\ref{5}) we have, via
(\ref{18a}) and (\ref{24}),
\begin{equation}
\begin{aligned}
\varepsilon^2 a_0 \int\limits_{\Omega_0^\varepsilon} \nabla
u^\varepsilon \cdot \nabla (g^2 \widetilde{u}^\varepsilon) \,dx +
\int\limits_{\widetilde{\Omega}_0^\varepsilon} \widetilde{a}_0
\nabla v^\varepsilon \cdot \nabla (g^2 \widetilde{u}^\varepsilon)
\,dx + \int\limits_{\mathbb{R}^n \backslash \Omega_0^\varepsilon}
a(x,\varepsilon) |\nabla (g \widetilde{u}^\varepsilon)|^2 \,dx -
\qquad
\\
- a_1 \int\limits_{\Omega_1^\varepsilon} |\nabla g|^2
(\widetilde{u}^\varepsilon)^2 \,dx - \lambda_\varepsilon
\int\limits_{\Omega_0^\varepsilon\cup\Omega_1^\varepsilon} g^2
(\widetilde{u}^\varepsilon)^2 \,dx -\lambda_\varepsilon
\int\limits_{\Omega_0^\varepsilon} g^2 v^\varepsilon
\widetilde{u}^\varepsilon \,dx  = \qquad
\\
= \lambda_\varepsilon \int\limits_{\widetilde{\Omega}_0^\varepsilon}
g^2 u^\varepsilon \widetilde{u}^\varepsilon \,dx  +
\lambda_\varepsilon \int\limits_{\Omega_2^\varepsilon} g^2
(\widetilde{u}^\varepsilon)^2 dx  +
\int\limits_{\Omega_2^\varepsilon\cup\widetilde{\Omega}_0^\varepsilon}
a(x,\varepsilon) |\nabla g|^2 (\widetilde{u}^\varepsilon)^2 \,dx.
\label{4.12}
\end{aligned}
\end{equation}
Notice that the right hand side is bounded by some constant $C$
independent of $\varepsilon$ and $R$ due to (\ref{12}), (\ref{130}),
(\ref{4.6}) and the boundedness of the domains of integration.

We employ (\ref{130}), (\ref{4.6}) and the boundedness of
$\widetilde{a}_0$ to conclude that the second term on the left hand
side of (\ref{4.12}) tends to zero (uniformly in $R$):
\begin{equation}
\begin{aligned}
\left| \, \int\limits_{\widetilde{\Omega}_0^\varepsilon}
\widetilde{a}_0 \nabla v^\varepsilon \cdot \nabla (g^2
\widetilde{u}^\varepsilon) \,dx \right| \leq C
\|\widetilde{u}^\varepsilon
\|_{L^2(\widetilde{\Omega}_0^\varepsilon)}  \rightarrow 0,
\label{26}
\end{aligned}
\end{equation}
as follows. Let us take an arbitrary subsequence
$\widetilde{u}^\varepsilon$. Since $\| \widetilde{u}^\varepsilon
\|_{H^1(\mathbb{R}^n)}$ is bounded uniformly in $\varepsilon$, see
(\ref{130}), the set of functions $\widetilde{u}^\varepsilon$ is
weakly compact in $H^1(B_R)$, hence strongly compact in $L^2(B_R)$
for any $R$; we take $R$ large enough so that $\Omega_2
\subset\subset B_R$. Then there exists further subsequence
$\widetilde{u}^\varepsilon$ that converges to some function $u_0$
strongly in $L^2(B_R)$. Then
\[
\|\widetilde{u}^\varepsilon
\|_{L^2(\widetilde{\Omega}_0^\varepsilon)} \leq  \| u_0
\|_{L^2(\widetilde{\Omega}^\varepsilon_0)} +  \|
\widetilde{u}^\varepsilon - u_0
\|_{L^2(\widetilde{\Omega}^\varepsilon_0)} \rightarrow 0
\]
as Lebesgue measure of the set $\widetilde{\Omega}^\varepsilon_0$
tends to zero. Since we have chosen in the beginning an arbitrary
subsequence $\widetilde{u}^\varepsilon$, (\ref{26}) follows. From
(\ref{4.6}) and (\ref{26}) we also obtain
\begin{equation}
\|v^\varepsilon\|_{L^2(\widetilde{\Omega}_0^\varepsilon)}
\rightarrow 0. \label{27}
\end{equation}

\textbf{Step 4.} The following Lemma approximates and bounds the
last and the first terms (both in a sense of a `two-scale' nature)
on the left hand side of (\ref{4.12}).
\begin{lemma}\label{lemma} There exists $\varepsilon_0 > 0$ such
that for all positive $\varepsilon < \varepsilon_0$ the following
estimates are valid
\begin{equation}
\begin{aligned}
\left|\lambda_\varepsilon \int\limits_{\Omega_0^\varepsilon} g^2
v^\varepsilon \widetilde{u}^\varepsilon \,dx -
(\beta(\lambda_\varepsilon) - \lambda_\varepsilon)
\int\limits_{\Omega_0^\varepsilon\cup\Omega_1^\varepsilon} g^2
(\widetilde{u}^\varepsilon)^2 \,dx \right| \leq
\\
\leq C \, \varepsilon  \left( \left\| \nabla(g
\widetilde{u}^\varepsilon) \right\|_{L^2(\Omega_1^\varepsilon)}^2 +
\| g
\widetilde{u}^\varepsilon\|_{L^2(\Omega_0^\varepsilon\cup\Omega_1^\varepsilon)}^2
\right)+ C, \label{4.13}
\end{aligned}
\end{equation}
and
\begin{equation}
\left| \varepsilon^2 a_0 \int\limits_{\Omega_0^\varepsilon} \nabla
u^\varepsilon \nabla (g^2 \widetilde{u}^\varepsilon) \,dx \right|
\leq C \, \varepsilon  \left( \| \nabla(g \widetilde{u}^\varepsilon)
\|_{L^2(\Omega_1^\varepsilon)}^2 + \| g
\widetilde{u}^\varepsilon\|_{L^2(\Omega_0^\varepsilon\cup\Omega_1^\varepsilon)}^2+
C \right), \label{4.14}
\end{equation}
where $C$ does not depend on $\varepsilon$ and $R$.
\end{lemma}
The proof of this lemma is quite technical and we give it in the
next section. We make use of Lemma \ref{lemma} and convergence
(\ref{26}) to transform identity (\ref{4.12}) into the following
inequality, valid for small enough $\varepsilon$:
\[
\begin{aligned}
& a_1 \| \nabla(g
\widetilde{u}^\varepsilon)\|_{L^2(\Omega_1^\varepsilon)}^2 - a_1 \|
(\nabla g) \widetilde{u}^\varepsilon\|_{L^2(\Omega_1^\varepsilon)}^2
- \beta(\lambda_\varepsilon) \| g \widetilde{u}^\varepsilon
\|_{L^2(\Omega_0^\varepsilon\cup\Omega_1^\varepsilon)}^2 -
\\
& - 2\delta  \left( \| \nabla(g \widetilde{u}^\varepsilon)
\|_{L^2(\Omega_1^\varepsilon)}^2 + \| g
\widetilde{u}^\varepsilon\|_{L^2(\Omega_0^\varepsilon\cup\Omega_1^\varepsilon)}^2\right)
\leq C,
\end{aligned}
\]
where C is independent of $\varepsilon$ and $R$. Notice that
$\beta(\lambda_\varepsilon)$ is negative and uniformly bounded away
from zero as $\lambda _\varepsilon \to \lambda _0$. Applying
(\ref{4.11}) to the second term on the left hand side we arrive at
\begin{equation}
(a_1 - 2\delta) \| \nabla(g
\widetilde{u}^\varepsilon)\|_{L^2(\Omega_1^\varepsilon)}^2 +
\left(-\beta(\lambda_\varepsilon)- \alpha^2 a_1 - 2\delta \right)
\| g \widetilde{u}^\varepsilon
\|_{L^2(\Omega_0^\varepsilon\cup\Omega_1^\varepsilon)}^2 \leq C,
\label{4.15}
\end{equation}
where $\delta>0$ could be chosen arbitrarily small. Hence we should
choose $\alpha$ such that  $-\beta(\lambda_0)- \alpha^2 a_1$ is
positive, i.e.
\[
\alpha < \sqrt{-\beta(\lambda_0)/a_1}.
\]
Since $g(|x|)$ coincides with $e^{\alpha|x|}$ on the ball $B_R$,
taking $\delta$ small enough and restricting the $L^2$-norms to
$B_R$ we arrive at
\[
\begin{aligned}
\left\| e^{\alpha|x|} \widetilde{u}^\varepsilon \right\|_{L^2(B_R)}
\leq C
\end{aligned}
\]
uniformly for small enough $\varepsilon$. Then passing to the limit
as $R\to \infty$ we obtain
\begin{equation}
\begin{aligned}
\left\| e^{\alpha|x|} \widetilde{u}^\varepsilon
\right\|_{L^2(\mathbb{R}^n)} \leq C. \label{4.16}
\end{aligned}
\end{equation}

\textbf{Step 5.} Now we easily get the same estimate for the
function $u^\varepsilon$:
\[
\begin{aligned}
\left\| e^{\alpha|x|} u^\varepsilon \right\|_{L^2(\mathbb{R}^n)}
\leq \left\| e^{\alpha|x|} \widetilde{u}^\varepsilon
\right\|_{L^2(\mathbb{R}^n)} + \sum_{\varepsilon Q_0 \subset
\Omega_0^\varepsilon\cup\widetilde{\Omega}_0^\varepsilon} \left\|
e^{\alpha|x|} v^\varepsilon \right\|_{L^2(\varepsilon Q_0)}.
\end{aligned}
\]
In each cell we use inequality (\ref{4.6}) and
\[
\sup_{x'\in \varepsilon Q} e^{\alpha |x'|} \leq e^{\alpha
\sqrt{n}\varepsilon} e^{\alpha |x|}, \qquad \forall x\in \varepsilon
Q,
\]
to obtain
\[
\left\| e^{\alpha|x|} v^\varepsilon \right\|_{L^2(\varepsilon Q_0)}
\leq C e^{\alpha \sqrt{n} \varepsilon} \left\| e^{\alpha|x|}
\widetilde{u}^\varepsilon \right\|_{L^2(\varepsilon Q_0)}  \leq C
\left\| e^{\alpha|x|} \widetilde{u}^\varepsilon
\right\|_{L^2(\varepsilon Q_0)},
\]
and hence, finally,
\[
\left\| e^{\alpha|x|} u^\varepsilon \right\|_{L^2(\mathbb{R}^n)}
\leq C
\]
uniformly in $\varepsilon$.
\end{proof}

\begin{remark} From (\ref{4.1}), (\ref{4.11}) and (\ref{4.15}) it also follows
that the gradient of $\widetilde{u}^\varepsilon$ decays
exponentially at infinity,
\begin{equation}
\begin{aligned}
\left\| e^{\alpha|x|} \nabla \widetilde{u}^\varepsilon
\right\|_{L^2(\mathbb{R}^n )} \leq C \label{28}
\end{aligned}
\end{equation}
uniformly in $\varepsilon$.
\end{remark}

\begin{remark} Estimate (\ref{15a}) is sharp in a sense. As we will
show later, $u_\varepsilon $ strongly two-scale converges to $u_0$,
for which $\sqrt{-\beta(\lambda_0)/a_1}$ is the optimal estimate for
its decay exponent, cf. (\ref{92a}).
\end{remark}

\section{Proof of Lemma \ref{lemma}.}\label{s4}
\begin{proof}

\textbf{Step 1.} First we decompose the function $v^\varepsilon $ in
$\Omega_0^\varepsilon $ into the sum of two functions:
\begin{equation}
v^\varepsilon = \widetilde{v}^\varepsilon + \widehat{v}^\varepsilon,
\end{equation}
solving the following equations (cf. (\ref{4.2})):
\begin{equation}
- a_0 \Delta_y \widetilde{v}^\varepsilon(\varepsilon y) -
\lambda_\varepsilon \widetilde{v}^\varepsilon(\varepsilon y) =
\lambda_\varepsilon \langle\widetilde{u}^\varepsilon(\varepsilon
y)\rangle_y, \,\, y\in Q_0, \quad
\widetilde{v}^\varepsilon(\varepsilon y) = 0, \,\, y\in
\partial Q_0, \label{39}
\end{equation}
\begin{equation}
- a_0 \Delta_y \widehat{v}^\varepsilon(\varepsilon y) -
\lambda_\varepsilon \widehat{v}^\varepsilon(\varepsilon y) =
\lambda_\varepsilon \left(\widetilde{u}^\varepsilon(\varepsilon y) -
\langle\widetilde{u}^\varepsilon(\varepsilon y)\rangle_y \right),
\,\, y\in Q_0, \quad \widehat{v}^\varepsilon(\varepsilon y) = 0,
\,\, y\in
\partial Q_0.  \label{40}
\end{equation}
The solution of (\ref{39}) could by presented in the form
\begin{equation}
\widetilde{v}^\varepsilon (\varepsilon y) = \lambda_\varepsilon
\langle\widetilde{u}^\varepsilon \rangle_y b_\varepsilon(y),
\label{41}
\end{equation}
where $b_\varepsilon $ is a solution of (\ref{13a}) with $\lambda =
\lambda_\varepsilon$. Due to the uniform (with respect to
$\varepsilon $) boundedness of the resolvent of the operator $T$ in
the neighborhood of $\lambda _0$, the solution of (\ref{40}) is
bounded as follows,
\begin{equation*}%\label{43}
\left\|\widehat{v}^\varepsilon (\varepsilon y)\right\|_{H^1(Q_0)}
\leq C \left\|\widetilde{u}^\varepsilon(\varepsilon y) -
\langle\widetilde{u}^\varepsilon\rangle_y \right\|_{L^2(Q_0)} \leq C
\left\|\nabla_y \widetilde{u}^\varepsilon \right\|_{L^2(Q_0)},
\end{equation*}
here we also employed the Poincar\'{e} inequality. In particular
\begin{equation}\label{43}
\|\widehat{v}^\varepsilon (x)\|_{L^2(\varepsilon Q_0)}  \leq
\varepsilon \, C \|\nabla
\widetilde{u}^\varepsilon(x)\|_{L^2(\varepsilon Q)},
\end{equation}
where $C$ in the inequality does not depend on $\varepsilon $ or
$\xi \in \mathbb Z^n$.

\textbf{Step 2.} At this stage we will need several inequalities
which follow from the properties of $g$ and
$\widetilde{u}^\varepsilon$.
\begin{proposition}\label{prop}The following estimates are valid for small enough
$\varepsilon$ with constants independent of $\,\varepsilon$ and the
choice of particular $\,\varepsilon Q$:
\begin{equation}
\begin{aligned}
\left\| g^2 \widetilde{u}^\varepsilon\right\|_{L^2(\varepsilon Q)}
\| \nabla \widetilde{u}^\varepsilon\|_{L^2(\varepsilon Q)} \leq  C
\left( \| \nabla(g \widetilde{u}^\varepsilon) \|_{L^2(\varepsilon
Q_1)}^2 + \| g \widetilde{u}^\varepsilon\|_{L^2(\varepsilon Q)}^2
\right), \label{n4.20}
\end{aligned}
\end{equation}
\begin{equation}
\begin{aligned}
\| \widetilde{u}^\varepsilon\|_{L^2(\varepsilon Q)} \left\| \nabla
(g^2 \widetilde{u}^\varepsilon)\right\|_{L^2(\varepsilon Q)} \leq  C
\left( \| \nabla(g \widetilde{u}^\varepsilon) \|_{L^2(\varepsilon
Q_1)}^2 + \| g \widetilde{u}^\varepsilon\|_{L^2(\varepsilon Q)}^2
\right),\label{n4.21}
\end{aligned}
\end{equation}
\begin{equation}
\begin{aligned}
\| \nabla \widetilde{u}^\varepsilon\|_{L^2(\varepsilon Q)} \left\|
\nabla (g^2 \widetilde{u}^\varepsilon)\right\|_{L^2(\varepsilon Q)}
\leq C \left( \| \nabla(g \widetilde{u}^\varepsilon)
\|_{L^2(\varepsilon Q_1)}^2 + \| g
\widetilde{u}^\varepsilon\|_{L^2(\varepsilon Q)}^2
\right).\label{n4.22}
\end{aligned}
\end{equation}
\end{proposition}
\begin{proof} Notice that
\begin{equation}
\begin{aligned}
\sup_{\varepsilon Q} g \leq e^{\alpha \sqrt{n}\varepsilon} g(x),
\qquad x\in \varepsilon Q. \label{4.32}
\end{aligned}
\end{equation}
We apply (\ref{4.1}), (\ref{4.11}) and (\ref{4.32}) to get
(\ref{n4.20}):
\[
\begin{aligned}
\left\| g^2 \widetilde{u}^\varepsilon\right\|_{L^2(\varepsilon Q)}
\left\| \nabla \widetilde{u}^\varepsilon\right\|_{L^2(\varepsilon
Q)} \leq C \left\| g^2
\widetilde{u}^\varepsilon\right\|_{L^2(\varepsilon Q)} \left\|
\nabla \widetilde{u}^\varepsilon\right\|_{L^2(\varepsilon Q_1)} \leq
C \left\| g \widetilde{u}^\varepsilon\right\|_{L^2(\varepsilon Q)}
\left\| g \nabla \widetilde{u}^\varepsilon\right\|_{L^2(\varepsilon
Q_1)} =
\\
= C \left\| g \widetilde{u}^\varepsilon\right\|_{L^2(\varepsilon Q)}
\left\| \nabla(g \widetilde{u}^\varepsilon) - (\nabla g)
\widetilde{u}^\varepsilon \right\|_{L^2(\varepsilon Q_1)} \leq C
\left( \left\| g \widetilde{u}^\varepsilon\right\|_{L^2(\varepsilon
Q)} \left\| \nabla(g \widetilde{u}^\varepsilon)
\right\|_{L^2(\varepsilon Q_1)} + \left\| g
\widetilde{u}^\varepsilon\right\|_{L^2(\varepsilon Q)}^2 \right)
\leq
\\
 \leq C \left( \left\| \nabla(g \widetilde{u}^\varepsilon)
\right\|_{L^2(\varepsilon Q_1)}^2 + \left\| g
\widetilde{u}^\varepsilon\right\|_{L^2(\varepsilon Q)}^2 \right).
%\label{n4.20}
\end{aligned}
\]
The proof of (\ref{n4.21}) and (\ref{n4.22}) is analogous.
\end{proof}

Let us show that the entity $\int\limits_{\Omega_0^\varepsilon} g^2
\widehat{v}^\varepsilon \widetilde{u}^\varepsilon \,dx$ is
relatively small (compared to the first term on the right hand side
of (\ref{4.13})). Indeed, applying inequalities (\ref{43}) and
(\ref{n4.20}) in each cell we obtain
\begin{equation}\label{48}
\int\limits_{\Omega_0^\varepsilon} g^2 \widehat{v}^\varepsilon
\widetilde{u}^\varepsilon \,dx \leq  \sum_{\varepsilon Q_0 \subset
\Omega_0^\varepsilon} \left\|g^2 \widetilde{u}^\varepsilon
\right\|_{L^2(\varepsilon Q_0)}
\|\widehat{v}^\varepsilon\|_{L^2(\varepsilon Q_0)} \leq \sum
\varepsilon \, C \left( \| \nabla(g \widetilde{u}^\varepsilon)
\|_{L^2(\varepsilon Q_1)}^2 + \| g
\widetilde{u}^\varepsilon\|_{L^2(\varepsilon Q)}^2 \right).
\end{equation}
Considering sets
\[
\bigcup_{\varepsilon Q_0 \subset \Omega_0^\varepsilon} \varepsilon
Q\qquad \textrm{ and } \qquad \bigcup_{\varepsilon Q_0 \subset
\Omega_0^\varepsilon} \varepsilon Q_1,
\]
one can notice that they are ``nearly'' equal to
\[
\Omega_0^\varepsilon\cup\Omega_1^\varepsilon \qquad \textrm{ and }
\qquad \Omega_1^\varepsilon,
\]
respectively. Namely,
\[
\begin{aligned}
\Omega_0^\varepsilon\cup\Omega_1^\varepsilon = \left(
\bigcup_{\varepsilon Q_0 \subset \Omega_0^\varepsilon} \varepsilon Q
\right) \cup \Omega_{1,+}^\varepsilon \setminus
\Omega_{1,-}^\varepsilon,
\\
\Omega_1^\varepsilon = \left( \bigcup_{\varepsilon Q_0 \subset
\Omega_0^\varepsilon} \varepsilon Q_1 \right) \cup
\Omega_{1,+}^\varepsilon \setminus \Omega_{1,-}^\varepsilon,
\end{aligned}
\]
where
\[
\begin{aligned}
\Omega_{1,-}^\varepsilon = \bigcup_{\varepsilon Q_0 \subset
\Omega_0^\varepsilon} \varepsilon Q \cap \Omega_2,
\\
\Omega_{1,+}^\varepsilon = \bigcup_{\varepsilon Q_0 \cap \Omega_2
\neq \emptyset} \varepsilon Q \cap \Omega_1^\varepsilon.
\end{aligned}
\]
We introduce two `correctors'
\[
r^\varepsilon = \| \nabla(g \widetilde{u}^\varepsilon)
\|_{L^2(\Omega_{1,-}^\varepsilon)}^2 + \| g
\widetilde{u}^\varepsilon\|_{L^2(\Omega_{1,-}^\varepsilon)}^2,
\]
and
\[
r_1^\varepsilon = \| g
\widetilde{u}^\varepsilon\|_{L^2(\Omega_{1,+}^\varepsilon \cup
\Omega_{1,-}^\varepsilon) }^2.
\]
Then inequality (\ref{48}) transforms into
\begin{equation}\label{49}
\int\limits_{\Omega_0^\varepsilon} g^2 \widehat{v}^\varepsilon
\widetilde{u}^\varepsilon \,dx \leq \varepsilon C \left( \| \nabla(g
\widetilde{u}^\varepsilon) \|_{L^2(\Omega_1^\varepsilon)}^2 + \| g
\widetilde{u}^\varepsilon\|_{L^2(\Omega_0^\varepsilon\cup\Omega_1^\varepsilon)}^2
+ r^\varepsilon \right).
\end{equation}

\textbf{Step 3.} Now we consider the term
$\int\limits_{\Omega_0^\varepsilon} g^2 \widetilde{v}^\varepsilon
\widetilde{u}^\varepsilon \,dx$ (cf. (\ref{4.13})) using also
(\ref{41}) and (\ref{11}):
\begin{equation}\label{50}
\begin{aligned}
&\left|\lambda_\varepsilon \int\limits_{\Omega_0^\varepsilon} g^2
\widetilde{v}^\varepsilon \widetilde{u}^\varepsilon \,dx -
(\beta(\lambda_\varepsilon) - \lambda_\varepsilon)
\int\limits_{\Omega_0^\varepsilon\cup\Omega_1^\varepsilon} g^2
(\widetilde{u}^\varepsilon)^2 \,dx \right| \leq
\\
&\leq C \varepsilon^n \sum_{\varepsilon Q_0 \subset
\Omega_0^\varepsilon} \left| \int\limits_Q g^2
\widetilde{u}^\varepsilon (\varepsilon y) b_\varepsilon(y) \langle
\widetilde{u}^\varepsilon \rangle_y \,dy - \langle b_\varepsilon
\rangle_y \int\limits_Q g^2(\varepsilon y)
(\widetilde{u}^\varepsilon(\varepsilon y))^2 \,dy \right| + C\,
r^\varepsilon_1 \leq
\\
&\leq C \varepsilon^n \sum \left| \langle \widetilde{u}^\varepsilon
\rangle_y \int\limits_Q \left(g^2 \widetilde{u}^\varepsilon -
\langle g^2 \widetilde{u}^\varepsilon \rangle_y \right)
b_\varepsilon \,dy\right| + \left| \langle b_\varepsilon \rangle_y
\int\limits_Q \left(g^2 \widetilde{u}^\varepsilon - \langle g^2
\widetilde{u}^\varepsilon \rangle_y \right)
\widetilde{u}^\varepsilon \,dy \right| + C\, r^\varepsilon_1.
\end{aligned}
\end{equation}
Notice that the mean value of $\widetilde{u}^\varepsilon$ is bounded
by its norm in $L^2$
\begin{equation}
\left|\langle \widetilde{u}^\varepsilon(\varepsilon y)
\rangle_y\right| = \left| \int\limits_{Q} \widetilde{u}^\varepsilon
\,dy \right| \leq \|\widetilde{u}^\varepsilon(x)\|_{L^2(Q)}.
\label{51}
\end{equation}
Similarly,
\begin{equation}
\langle b_\varepsilon \rangle_y \leq
\|b_\varepsilon \|_{L^2(Q_0)} \leq C, %\label{1.25}
\end{equation}
where $C$ does not depend on $\varepsilon $ due to the uniform
boundedness of $(T-\lambda )^{-1}$ in the neighborhood of
$\lambda_0$. Via the Poincar\'{e} inequality we derive
\begin{equation}
\left|\int\limits_Q \left(g^2 \widetilde{u}^\varepsilon - \langle
g^2 \widetilde{u}^\varepsilon \rangle_y \right)
\widetilde{u}^\varepsilon \,dy \right| \leq \left\| g^2
\widetilde{u}^\varepsilon - \langle g^2 \widetilde{u}^\varepsilon
\rangle_y \right\|_{L^2(Q)} \left\| \widetilde{u}^\varepsilon
\right\|_{L^2(Q)} \leq C \left\|\nabla_y \left( g^2
\widetilde{u}^\varepsilon \right) \right\|_{L^2(Q)} \left\|
\widetilde{u}^\varepsilon \right\|_{L^2(Q)},
\end{equation}
and
\begin{equation}\label{54a}
\left|\int\limits_Q \left(g^2 \widetilde{u}^\varepsilon - \langle
g^2 \widetilde{u}^\varepsilon \rangle_y \right) b_\varepsilon \,dy
\right| \leq C \left\|\nabla_y \left( g^2 \widetilde{u}^\varepsilon
\right) \right\|_{L^2(Q)},
\end{equation}
with constants independent of $\varepsilon $ and $\xi$. Applying
inequalities (\ref{51})--(\ref{54a}) and then (\ref{n4.21}) to
(\ref{50}) we arrive at
\begin{equation}\label{55}
\begin{aligned}
\left|\lambda_\varepsilon \int\limits_{\Omega_0^\varepsilon} g^2
\widetilde{v}^\varepsilon \widetilde{u}^\varepsilon \,dx +
(\lambda_\varepsilon - \beta(\lambda_\varepsilon))
\int\limits_{\Omega_0^\varepsilon\cup\Omega_1^\varepsilon} g^2
(\widetilde{u}^\varepsilon)^2 \,dx \right| \leq
\\
\leq \varepsilon \,C \left( \| \nabla(g \widetilde{u}^\varepsilon)
\|_{L^2(\Omega_1^\varepsilon)}^2 + \| g
\widetilde{u}^\varepsilon\|_{L^2(\Omega_0^\varepsilon\cup\Omega_1^\varepsilon)}^2
+ r^\varepsilon \right) + C\, r^\varepsilon_1,
\end{aligned}
\end{equation}
where $C$ is $\varepsilon $-independent. Since the correctors
$r^\varepsilon, \, r^\varepsilon_1 $ are uniformly bounded,
inequalities (\ref{49}) and (\ref{55}) together imply the validity
of (\ref{4.13}).

\textbf{Step 4.} Finally, it is not difficult to obtain similarly
(\ref{4.14}) via (\ref{4.6}), (\ref{n4.21}) and (\ref{n4.22}):
\[
\begin{aligned}
&\left| \varepsilon^2 a_0 \int\limits_{\Omega_0^\varepsilon} \nabla
u^\varepsilon \nabla (g^2 \widetilde{u}^\varepsilon) \,dx \right|
\leq \varepsilon^2 C  \sum_{Q_0:\varepsilon Q_0 \subset
\Omega_0^\varepsilon} \| \nabla u^\varepsilon\|_{L^2(\varepsilon
Q_0)} \left\| \nabla(g^2 \widetilde{u}^\varepsilon)
\right\|_{L^2(\varepsilon Q_0)} \leq
\\
\leq &\varepsilon C  \sum_{Q_0} \big( \|\varepsilon \nabla
v^\varepsilon\|_{L^2(\varepsilon Q_0)} + \varepsilon \| \nabla
\widetilde{u}^\varepsilon\|_{L^2(\varepsilon Q_0)} \big) \left\|
\nabla(g^2 \widetilde{u}^\varepsilon) \right\|_{L^2(\varepsilon
Q_0)} \leq
\\
\leq & \varepsilon C \sum_{Q_0} \left( \|
\widetilde{u}^\varepsilon\|_{L^2(\varepsilon Q_0)}\left\| \nabla(g^2
\widetilde{u}^\varepsilon) \right\|_{L^2(\varepsilon Q_0)} +
\varepsilon \| \nabla \widetilde{u}^\varepsilon\|_{L^2(\varepsilon
Q_0)} \left\| \nabla(g^2 \widetilde{u}^\varepsilon)
\right\|_{L^2(\varepsilon Q_0)} \right) \leq
\\
\leq & \delta  \left( \| \nabla(g \widetilde{u}^\varepsilon)
\|_{L^2(\Omega_1^\varepsilon)}^2 + \| g
\widetilde{u}^\varepsilon\|_{L^2(\Omega_0^\varepsilon\cup\Omega_1^\varepsilon)}^2+
r^\varepsilon \right) \leq  \delta  \left( \| \nabla(g
\widetilde{u}^\varepsilon) \|_{L^2(\Omega_1^\varepsilon)}^2 + \| g
\widetilde{u}^\varepsilon\|_{L^2(\Omega_0^\varepsilon\cup\Omega_1^\varepsilon)}^2+
C \right)
\end{aligned}
\]
for small enough $\varepsilon $.

Notice that all the estimates obtained in this section are
independent of $R$.
\end{proof}

\section{Some properties of two-scale convergence}\label{s5}
In this section we list the definitions and some properties of the
two-scale convergence, see \cite{Al, Ngu, Zhikov2, Zhikov1}. We also
formulate several statements (analogous to those in \cite{Zhikov2})
which are necessary for obtaining the two-scale convergence of the
eigenfunctions of $A_\varepsilon$ and derivation of the limit
equation.

Let $\Omega$ be an arbitrary region in $\mathbb{R}^n$, in particular
$\Omega = \mathbb{R}^n$. Denote by $\square$ the unit cube
$[0,1)^n$. We consider all functions of the form $u(x,y)$ to be
1-periodic in $y$ in each coordinate.

\begin{definition} We say that bounded in
$L^2(\Omega)$ sequence $v_\varepsilon$ is weakly two-scale
convergent to a function $v\in L^2(\Omega\times \square)$,
$v_\varepsilon(x) \stackrel{2}\rightharpoonup v(x,y)$, if
\[
\lim_{\varepsilon\to 0} \int\limits_{\Omega} v_\varepsilon(x)
\varphi (x) b \left(\frac{x}{\varepsilon}\right) \,dx =
\int\limits_{\Omega}\int\limits_{\square} v(x,y) \varphi(x) b(y)
\,dy dx
\]
for all $\varphi \in C_0^\infty(\Omega)$ and all $b\in
C^\infty_{\mathrm{per}}(\square)$ (where
$C^\infty_{\mathrm{per}}(\square)$ is the set of 1-periodic
functions from $C^\infty(\mathbb{R}^n)$).
\end{definition}

\begin{definition} We say that a
bounded in $L^2(\Omega)$ sequence $u_\varepsilon$ is strongly
two-scale convergent to a function $u \in L^2(\Omega\times
\square)$, $u_\varepsilon(x) \stackrel{2}\rightarrow u(x,y)$, if
\[
\lim_{\varepsilon\to 0} \int\limits_{\Omega} u_\varepsilon(x)
v_\varepsilon(x) \,dx = \int\limits_{\Omega}\int\limits_{\square}
u(x,y) v(x,y) \,dy \,dx
\]
for all $v_\varepsilon(x) \stackrel{2}\rightharpoonup v(x,y)$.
\end{definition}

\begin{proposition}(Properties of the two-scale convergence.) \label{p5.3}
\\
(i) If $u_\varepsilon(x) \stackrel{2}\rightharpoonup u(x,y)$ and
$a\in L^\infty_{\mathrm{per}}(\square)$ then
\[
a(x/\varepsilon) u_\varepsilon(x) \stackrel{2}\rightharpoonup a(y)
u(x,y).
\]
\\
(ii) $v_\varepsilon(x) \stackrel{2}\rightarrow v(x,y)$ if and only
if $v_\varepsilon(x) \stackrel{2}\rightharpoonup v(x,y)$ and
\[
\lim_{\varepsilon\to 0} \int\limits_{\Omega} v_\varepsilon^2 \,dx =
\int\limits_{\Omega}\int\limits_{\square} v^2 \,dy \,dx.
\]
\\
(iii) If $f_\varepsilon(x) \rightarrow f(x)$ in $L^2(\Omega)$, then
$f_\varepsilon(x) \stackrel{2}\rightarrow f(x)$.
\end{proposition}

\begin{proposition}(The mean value property of periodic functions.) Let
$\Phi(y)\in L^1_{\mathrm{per}}(\square)$ . Then for each $\phi(x)\in
C^\infty_0 (\mathbb{R}^n)$ we have
\[
\lim_{\varepsilon \rightarrow 0} \int\limits_{\mathbb{R}^n}
\phi(x)\Phi(x/\varepsilon)dx = \langle \Phi \rangle_y
\int\limits_{\mathbb{R}^n} \phi(x)dx.
\]
\end{proposition}

Potential vector space $V_{\mathrm{pot}}$ is defined as a closure of
the set $\{\nabla \varphi : \varphi \in
C^\infty_{\mathrm{per}}(\square)\}$ in $L^2(\square)^n$. We say that
a vector $b\in L^2(\square)^n$ is solenoidal ($b \in
V_{\mathrm{sol}}$) if it is orthogonal to all potential vectors.
Thus,
\[
L^2(\square)^n = V_{\mathrm{pot}} \oplus V_{\mathrm{sol}},
\]
and
\[
L^2(\Omega\times\square)^n = L^2(\Omega, V_{\mathrm{pot}}) \oplus
L^2(\Omega, V_{\mathrm{sol}}). \label{6.1}
\]

\begin{lemma}\label{L5.1}
Let $u_\varepsilon$ and $\varepsilon \nabla u_\varepsilon$ be
bounded in $L^2(\mathbb{R}^n)$. Then (up to a subsequence)
\[
\begin{aligned}
u_\varepsilon(x) \stackrel{2}\rightharpoonup u(x,y) \in
L^2(\mathbb{R}^n, H^1_{\mathrm{per}}),
\\
\varepsilon \nabla u_\varepsilon(x) \stackrel{2}\rightharpoonup
\nabla_y u(x,y), \qquad
\end{aligned}
\]
where $H^1_{\mathrm{per}} = H^1_{\mathrm{per}}(\square)$ is the
Sobolev space of periodic functions.
\end{lemma}

\begin{lemma}\label{L07}
Let $u_\varepsilon\in H^1(\mathbb{R}^n)$,
\begin{equation}
\begin{aligned}
u_\varepsilon(x) \stackrel{2}\rightharpoonup u(x)\in
H^1(\mathbb{R}^n),\label{6.2}
\end{aligned}
\end{equation}
and $\nabla u_\varepsilon$ is bounded in $L^2(\mathbb{R}^n)$. Then,
up to a subsequence,
\begin{equation}
\begin{aligned}
\nabla u_\varepsilon(x) \stackrel{2}\rightharpoonup \nabla u(x) +
v(x,y), \mbox{ where } v\in L^2(\mathbb{R}^n, V_{\mathrm{pot}}).
\label{6.3}
\end{aligned}
\end{equation}
\end{lemma}

\begin{lemma}\label{L08}
Let (\ref{6.2}) and (\ref{6.3}) be valid. Let also
\begin{equation}
\begin{aligned}
\lim_{\varepsilon \to 0} \int\limits_{\Omega_1^\varepsilon} a_1
\nabla u_\varepsilon(x) \cdot \nabla_y w(\varepsilon ^{-1} x)
\varphi(x) \,dx = 0
\end{aligned}\label{54}
\end{equation}
for any $\varphi \in C_0^\infty(\Omega_1)$ and $w\in
C^\infty_{\mathrm{per}}(\square)$. Then the following weak
convergence of the flows takes place:
\[
\begin{aligned}
a_1 \theta_{Q_1}(\varepsilon^{-1} x) \nabla u_\varepsilon(x)
\rightharpoonup A^{\mathrm{hom}} \nabla u(x) \mbox{ in } \Omega_1,
\label{6.4}
\end{aligned}
\]
where homogenized matrix $A^{\rm{hom}}$ is defined by
(\ref{eq:ahom}).
\end{lemma}

The proofs of the listed statements repeat the proofs of the
corresponding assertions in \cite{Zhikov2} with no or only small
alterations, and are not given here.

\begin{definition}\label{defrc}
Let $A_\varepsilon$, $\varepsilon >0$, and $A_0$ be non-negative
self-adjoint operators in $L^2(\mathbb{R}^n)$ and $\mathcal {H}_0
\subset L^2(\mathbb{R}^n \times Q)$, see (\ref{5.1}), respectively.
We say that $A_\varepsilon \stackrel{2}\rightarrow A_0$ in the sense
of the strong two-scale resolvent convergence if
$\left(A_\varepsilon + I \right)^{-1} f_\varepsilon
\stackrel{2}\rightarrow \left(A_0 + I \right)^{-1} f_0$ as long as
$f_\varepsilon \stackrel{2}\rightarrow f_0$.
\end{definition}

\section{Strong two-scale convergence of the eigenfunctions and multiplicity of the eigenvalues of $A_\varepsilon$}\label{s6}

In this section we will show that the normalized eigenfunctions
$u_\varepsilon$ are compact in the sense of strong two-scale
convergence. Namely, provided $\lambda_\varepsilon \to \lambda_0$, a
sequence of normalized eigenfunctions $u_\varepsilon$ of the
operator $A_\varepsilon$ strongly two-scale converges, up to a
subsequence, to a function $u^0(x,y)$. This implies that $u^0(x,y)$
is an eigenfunction corresponding to the eigenvalue $\lambda_0$ of
the limit operator $A_0$. This, together with results of
\cite{SmysKam}, establishes an `asymptotic one-to-one
correspondence' between isolated eigenvalues and corresponding
eigenfunctions of the operators $A_\varepsilon$ and $A_0$.

\begin{theorem}\label{th6}
Under the assumptions of Theorem \ref{theorem1} $\lambda_0$ is an
eigenvalue of the operator $A_0$. Moreover, there exists a
subsequence $\varepsilon$ such that eigenfunctions $u^\varepsilon$
of the operator $A_\varepsilon$ strongly two-scale converge to an
eigenfunction $u^0(x,y)$ of $A_0$ corresponding to the eigenvalue
$\lambda_0$.
\end{theorem}
\begin{proof}

\textbf{Step 1.} In order to establish strong two-scale convergence
of the eigenfunctions $u^\varepsilon = \widetilde{u}^\varepsilon +
v^\varepsilon$ we establish it for each of its components
separately. From (\ref{4.16}) and (\ref{28}) it follows that
\begin{equation}
\begin{aligned}
\| \widetilde{u}^\varepsilon \|_{H^1(\mathbb{R}^n \backslash B_R)}
\leq C e^{-\alpha R} \label{59}
\end{aligned}
\end{equation}
with $C$ independent of $\varepsilon$ and $R$. From this one can
easily conclude that $\widetilde{u}^\varepsilon$ is weakly compact
in $H^1(\mathbb{R}^n)$ and strongly compact in $L^2(\mathbb{R}^n)$.
Indeed, since $\widetilde{u}^\varepsilon$ are bounded in
$H^1(\mathbb{R}^n)$ uniformly in $\varepsilon$,
\begin{equation}
\begin{aligned}
\widetilde{u}^\varepsilon \rightharpoonup u_0 \, \text{ in }\,
H^1(\mathbb{R}^n), \label{7.6}
\end{aligned}
\end{equation}
up to a subsequence. For any fixed $R$ function
$\widetilde{u}^\varepsilon$ converges to $u_0$ weakly in $H^1(B_R)$
and, hence, strongly in $L^2(B_R)$ up to a subsequence. Considering
a sequence of balls $B_R$, $R\in\mathbb{N}$, one can use the method
of extracting a diagonal subsequence such that
\begin{equation}
\widetilde{u}^\varepsilon \rightarrow u_0 \, \text{ in }\, L^2(B_R)
\label{57}
\end{equation}
for any $R>0$.

For any $\delta
> 0$ we can choose $R$ such that $\|u_0\|_{L^2(\mathbb{R}^n \backslash
B_R)} < \delta / 3$ and
$\|\widetilde{u}^\varepsilon\|_{L^2(\mathbb{R}^n \backslash B_R)} <
\delta / 3$ for sufficiently small $\varepsilon$ (the latter follows
from (\ref{59})). From (\ref{57}) it follows that $\|u_0 -
\widetilde{u}^\varepsilon\|_{L^2(B_R)} < \delta / 3$ for
sufficiently small $\varepsilon$. Then, up to a subsequence,
\[
\begin{aligned}
\|u_0 - \widetilde{u}^\varepsilon\|_{L^2(\mathbb{R}^n)} \leq \|u_0 -
\widetilde{u}^\varepsilon\|_{L^2(B_R)} + \|u_0\|_{L^2(\mathbb{R}^n
\backslash B_R)} + \|\widetilde{u}^\varepsilon\|_{L^2(\mathbb{R}^n
\backslash B_R)} < \delta
\end{aligned}
\]
for small enough $\varepsilon$. Hence, up to a subsequence, we have
\[
\begin{aligned}
\widetilde{u}^\varepsilon \rightarrow u_0 \, \text{ in }\,
L^2(\mathbb{R}^n). \label{7.8}
\end{aligned}
\]
Then from properties of the two-scale convergence we conclude that
\begin{equation}
\widetilde{u}^\varepsilon \stackrel{2}\rightarrow u_0.  \label{7.9}
\end{equation}

\textbf{Step 2.} Now let us consider  $v^\varepsilon$. We denote by
$v^\varepsilon_1$ and $v^\varepsilon_2$ its restrictions
$v_\varepsilon |_{\Omega_0^\varepsilon}$ and $v_\varepsilon
|_{\widetilde{\Omega}_0^\varepsilon}$ respectively, extended by zero
to the rest of $\mathbb{R}^n$.

\begin{lemma} \label{L7}The following convergence properties are valid for $v^\varepsilon_1$
(up to a subsequence):
\[
\begin{aligned}
v^\varepsilon_1 (x) &\stackrel{2}\rightarrow v(x,y) \in
L^2(\Omega_1, H^1_0(Q_0)),
\\
\varepsilon \nabla v^\varepsilon_1 (x) &\stackrel{2}\rightharpoonup
\nabla_y v(x,y),
\end{aligned}
\]
where $v(x,y)$ is a solution to the following problem:
\begin{equation}
-a_0 \Delta_y v - \lambda_0 v = \lambda_0 u_0, \quad y \in Q_0.
\label{7.22}
\end{equation}
Here $u_0$ is a function from (\ref{7.9}).
\end{lemma}
\begin{proof} Function $v^\varepsilon_1 \in H^1(\Omega_0^\varepsilon)$
satisfies the following differential equation:
\begin{equation}
-\varepsilon^2 a_0 \Delta v^\varepsilon_1 - \lambda_\varepsilon
v^\varepsilon_1 = \lambda_\varepsilon \widetilde{u}^\varepsilon
\,\,\, \mbox{in} \,\, \Omega^\varepsilon_0. \label{7.10}
\end{equation}
The right hand side of this equation is of the form
$\lambda_\varepsilon \theta_{\Omega^\varepsilon_0}
\widetilde{u}^\varepsilon$. By (\ref{7.9}) and the properties of the
two-scale convergence we have
\begin{equation}
\lambda_\varepsilon \theta_{\Omega^\varepsilon_0}(x)
\widetilde{u}^\varepsilon(x) \stackrel{2}\rightarrow \lambda_0
\theta_{Q_0}(y) \theta_{\Omega_1}(x) u_0(x). \label{7.13}
\end{equation}

Following \cite{Zhikov2} we consider more general problem
\begin{equation}
z_\varepsilon \in H^1(\Omega_0^\varepsilon), \quad -\varepsilon^2
a_0 \Delta z_\varepsilon - \lambda_\varepsilon z_\varepsilon =
f_\varepsilon, \quad f_\varepsilon \in
L^2(\Omega^\varepsilon_0).\label{7.12}
\end{equation}
(It is implicit that $f_\varepsilon=z_\varepsilon=0$ in
$\mathbb{R}^n \backslash \Omega^\varepsilon_0$.)
\begin{proposition}Let
\begin{equation}
f^\varepsilon (x) \stackrel{2}\rightharpoonup f(x,y).
\end{equation}
Then
\[
\begin{aligned}
z^\varepsilon (x) &\stackrel{2}\rightharpoonup z(x,y) \in
L^2(\Omega_1, H^1_0(Q_0)),
\\
\varepsilon \nabla z^\varepsilon (x) &\stackrel{2}\rightharpoonup
\nabla_y z(x,y), \label{7.14}
\end{aligned}
\]
where function $z(x,y)$ solves the following equation:
\begin{equation}
-a_0 \Delta_y z - \lambda_0 z = f, \quad y \in Q_0. \label{7.15}
\end{equation}
\end{proposition}
\begin{proof}
One can easily derive an estimate for $z^\varepsilon$ analogous to
(\ref{4.6}), applying to (\ref{7.12}) a reasoning similar to those
for the solution of equation (\ref{13}). This give us the weak
two-scale convergence of $z^\varepsilon$ and $\varepsilon \nabla
z^\varepsilon$ via Lemma \ref{L5.1}. The result follows by a
straightforward passing to the limit in the integral identity
corresponding to (\ref{7.12}) with appropriately chosen test
function. The full proof could be found in \cite{Zhikov2} and
applies to the present situation with no alteration.
\end{proof}

The above proposition together with (\ref{7.13}) establishes a
``weak'' form of the statement of the lemma, i.e. weak two-scale
convergence of $v^\varepsilon_1$. We now prove that the convergence
is actually strong, following again \cite{Zhikov2}. Multiply
(\ref{7.10}) and (\ref{7.12}) by $z^\varepsilon$ and
$v^\varepsilon_1$ respectively and integrate by parts. The left hand
sides of the resulting equalities are identical. So, equating the
right hand sides, we obtain the following identity
\[
\int\limits_{\Omega_1} f^\varepsilon v^\varepsilon \,dx =
\lambda_\varepsilon \int\limits_{\Omega_1} \widetilde{u}^\varepsilon
z^\varepsilon \,dx.
\]
By the definition of the strong two-scale convergence we have
\[
\lim_{\varepsilon \to 0} \lambda_\varepsilon \int\limits_{\Omega_1}
\widetilde{u}^\varepsilon z^\varepsilon \,dx = \lambda_0
\int\limits_{\Omega_1} \int\limits_{Q_0}u_0(x) z(x,y) \,dy \,dx.
\]
Multiplying (\ref{7.22}) and (\ref{7.15}) by $z$ and $v$
respectively and integrating by parts it is easy to see that
\[
\lambda_0 \int\limits_{\Omega_1} \int\limits_{Q_0}u_0(x) z(x,y) \,dy
dx = \int\limits_{\Omega_1} \int\limits_{Q_0}f(x,y) v(x,y) \,dy
\,dx.
\]
Thus, we have a convergence of the integrals:
\[
\lim_{\varepsilon \to 0} \int\limits_{\Omega_1} f^\varepsilon
v^\varepsilon_1 \,dx  = \int\limits_{\Omega_1}
\int\limits_{Q_0}f(x,y) v(x,y) \,dy \,dx
\]
for any weakly two-scale convergent sequence $f^\varepsilon$. Hence,
by the definition,
\[
v^\varepsilon_1 (x) \stackrel{2}\rightarrow v(x,y).
\]
\end{proof}

\begin{lemma} \label{L9}
Sequence of functions $v^\varepsilon_2$ converges to zero in the
sense of strong two-scale convergence:
\[
\begin{aligned}
v^\varepsilon_2 \stackrel{2}\rightarrow 0 \,\, \mbox{as }
\varepsilon \to 0.
\end{aligned}
\]
\end{lemma}
\begin{proof}
Straightforward from (\ref{27}) and Proposition \ref{p5.3}
\textit{(iii)}.
\end{proof}

Combining (\ref{7.9}) with Lemmas \ref{L7} and \ref{L9}, we arrive
at
\begin{equation}
\begin{aligned}
u^\varepsilon (x) \stackrel{2}\rightarrow u^0(x,y) = u_0(x) +
v(x,y),\label{7.37}
\end{aligned}
\end{equation}
where $u_0\in H^1(\mathbb{R}^n), \,\, v \in L^2(\Omega_1,
H^1_0(Q_0))$.

\medskip
\textbf{Step 3.} Now it remains to show that $u^0(x,y)$ is an
eigenfunction and $\lambda_0$ is the corresponding eigenvalue of the
limit operator $A_0$, i.e. that $u^0(x,y)$ satisfies (\ref{10}). In
order to do that we need to choose appropriate test-function
$\psi^\varepsilon$ and pass to the limit in the integral identity
\begin{equation}
\begin{aligned}
\varepsilon^2 a_0 \int\limits_{\Omega_0^\varepsilon} \nabla
u^\varepsilon \cdot \nabla \psi^\varepsilon \,dx + a_1
\int\limits_{\Omega_1^\varepsilon}& \nabla u^\varepsilon \cdot\nabla
\psi^\varepsilon \,dx +
\int\limits_{\widetilde{\Omega}_0^\varepsilon} \widetilde{a}_0
\nabla u^\varepsilon \cdot\nabla \psi^\varepsilon \,dx +
\\
&+ a_2 \int\limits_{\Omega_2^\varepsilon} \nabla u^\varepsilon \cdot
\nabla \psi^\varepsilon \,dx = \lambda_\varepsilon
\int\limits_{\mathbb{R}^n} u^\varepsilon \psi^\varepsilon \,dx
\label{7.38}
\end{aligned}
\end{equation}
corresponding to the original eigenvalue problem
(\ref{2})--(\ref{3}). Let us take
\begin{equation}
\begin{aligned}
&\psi^\varepsilon(x) = \psi_0(x) + \varphi(x) b(\varepsilon^{-1} x),
\\
&\psi_0 \in C_0^\infty(\mathbb{R}^n), \, \varphi \in
C_0^\infty(\Omega_1), \, b(y) \in C_0^\infty(Q_0), \label{7.39}
\end{aligned}
\end{equation}
and consider each term of (\ref{7.38}) separately. Let us expand the
first term:
\[
\begin{aligned}
\varepsilon^2 a_0 \int\limits_{\Omega_0^\varepsilon} \nabla
u^\varepsilon &\nabla \psi^\varepsilon \,dx = \varepsilon^2 a_0
\int\limits_{\Omega_0^\varepsilon} \nabla \widetilde{u}^\varepsilon
\nabla \psi^\varepsilon \,dx +
\\
+ \varepsilon^2 a_0 &\int\limits_{\Omega_0^\varepsilon} \nabla
v^\varepsilon \left(\nabla \psi_0 + b(\varepsilon^{-1} x) \nabla
\varphi\right) dx +  a_0
\int\limits_{\Omega_0^\varepsilon}\varepsilon \nabla v^\varepsilon \varphi \nabla_y b(\varepsilon^{-1} x) \,dx. % \label{7.38}
\end{aligned}
\]
As $\nabla \widetilde{u}^\varepsilon$ is bounded in $L^2$-norm and
$|\nabla \psi^\varepsilon|\leq C \varepsilon^{-1}$ then the first
term on the right hand side tends to zero. From (\ref{4.6}) and the
boundedness of $\nabla \psi_0 + b \nabla \varphi$ we conclude that
the second term also converges to zero. Since by Lemma \ref{L7}
$\varepsilon \nabla v^\varepsilon$ converges two-scale weakly, from
the definition of the weak two-scale convergence we obtain
\begin{equation}
\lim_{\varepsilon \to 0}\varepsilon^2 a_0
\int\limits_{\Omega_0^\varepsilon} \nabla u^\varepsilon \nabla
\psi^\varepsilon \,dx = a_0 \int\limits_{\Omega_1} \int\limits_{Q_0}
\nabla_y v(x,y) \varphi(x) \nabla_y b(y) \,dy \,dx. \label{72n}
\end{equation}

Let us show that convergence property (\ref{54}) holds for
$u^\varepsilon$. To this end we substitute into (\ref{7.38}) a test
function of the form $\varepsilon \, w(\varepsilon ^{-1} x)
\varphi(x)$, $\varphi \in C_0^\infty(\Omega_1)$, $w\in
C^\infty_{\mathrm{per}}(\square)$, cf. \cite{Zhikov2}. Then all the
terms except, possibly,
\[
\int\limits_{\Omega_1^\varepsilon} a_1 \nabla u^\varepsilon(x) \cdot
\nabla_y w(\varepsilon ^{-1} x) \varphi(x) \,dx
\]
converge to zero. As a result, the above term also converges to
zero. We then apply Lemma \ref{L07} for $u_\varepsilon$ replaced by
$\widetilde{u}^\varepsilon$. Since $\widetilde{u}^\varepsilon$
coincides with $u^\varepsilon$ on $\Omega_1^\varepsilon$, by Lemma
\ref{L08} applied to the second term on the left hand side of
(\ref{7.38}) with $\psi^\varepsilon$ as in (\ref{7.39}) we obtain
\begin{equation}
\lim_{\varepsilon \to 0}\, a_1 \int\limits_{\Omega_1^\varepsilon}
\nabla u^\varepsilon \cdot \nabla \psi^\varepsilon \,dx =
\lim_{\varepsilon \to 0}\, a_1 \int\limits_{\Omega_1^\varepsilon}
\nabla u^\varepsilon \cdot \nabla \psi_0 \,dx =
\int\limits_{\Omega_1} A^{\rm{hom}}\nabla u_0 \cdot \nabla \psi_0
\,dx. \label{73n}
\end{equation}

For small enough $\varepsilon$ the function $\psi^\varepsilon$ is
equal to $\psi_0$ in $\widetilde{\Omega}_0^\varepsilon$, so
$\nabla\psi^\varepsilon$ is bounded in $
\widetilde{\Omega}_0^\varepsilon$. Since
$\int\limits_{\widetilde{\Omega}_0^\varepsilon} \widetilde{a}_0
|\nabla u^\varepsilon|^2\, \,dx$ is bounded uniformly in
$\varepsilon$ and $\big|\widetilde{\Omega}_0^\varepsilon\big|\to 0$
as $\varepsilon \to 0$, we have
\begin{equation}
\left|\int\limits_{\widetilde{\Omega}_0^\varepsilon} \widetilde{a}_0
\nabla u^\varepsilon \nabla \psi^\varepsilon \,dx \right| \leq C
\int\limits_{\widetilde{\Omega}_0^\varepsilon} \widetilde{a}_0
|\nabla u^\varepsilon| \,dx \leq C
\big|\widetilde{\Omega}_0^\varepsilon\big|^{1/2}\,
\widetilde{a}_0^{1/2}\,
\left(\int\limits_{\widetilde{\Omega}_0^\varepsilon} \widetilde{a}_0
|\nabla u^\varepsilon|^2 \,dx\right)^{1/2} \rightarrow 0.
\label{74n}
\end{equation}

The function $u^\varepsilon$ coincides with
$\widetilde{u}^\varepsilon$ on $\Omega_2^\varepsilon$. Then, via
(\ref{7.6}) we have convergence of the last term on the left hand
side of (\ref{7.38}):
\begin{equation}
\lim_{\varepsilon \to 0} \,a_2 \int\limits_{\Omega_2^\varepsilon}
\nabla u^\varepsilon \cdot \nabla \psi^\varepsilon \,dx =
 \lim_{\varepsilon \to 0} \left[a_2 \int\limits_{\Omega_2} \nabla
\widetilde{u}^\varepsilon \cdot \nabla \psi_0 \,dx - a_2
\int\limits_{\widetilde{\Omega}_0^\varepsilon \cap \Omega_2} \nabla
\widetilde{u}^\varepsilon \cdot \nabla \psi_0 \,dx \right] =  a_2
\int\limits_{\Omega_2} \nabla u_0 \cdot \nabla \psi_0 \,dx.
\label{75n}
\end{equation}

Thus, passing to the limit as $\varepsilon \to 0$ on the left hand
side of (\ref{7.38}) via (\ref{72n})--(\ref{75n}), and on the right
hand side via (\ref{7.37}), we arrive at
\[
a_0 \int\limits_{\Omega_1} \int\limits_{Q_0} \nabla_y v \cdot
\varphi \nabla_y b \,dy \,dx + \int\limits_{\Omega_1}
A^{\rm{hom}}\nabla u_0\cdot \nabla \psi_0 \,dx + a_2
\int\limits_{\Omega_2} \nabla u_0 \cdot \nabla \psi_0 \,dx =
 \lambda_0 \int\limits_{\mathbb{R}^n} \int\limits_{Q} (u_0 + v)
(\psi_0 + \varphi\, b) \,dy \,dx .
\]
Since the space of functions from (\ref{7.39}) is dense in
$\mathcal{V}$ (see (\ref{Qdomain})), the latter is equivalent to
(\ref{10}). It follows from (\ref{7.37}), Proposition \ref{p5.3}
(\textit{ii}) and the normalization of $u^\varepsilon$ that
$u^0(x,y) \not\equiv 0$. Thus we have proved that $\lambda_0$ and
$u^0(x,y)$ are respectively an eigenvalue and an eigenfunction of
the operator $A_0$, completing the proof of the theorem.
\end{proof}

\begin{remark}
Theorem \ref{th6} combined with \cite[Theorem 2]{FigKlein} implies
the existence of eigenvalues of $A_0$ in the gaps of its essential
spectrum, provided $\Omega_2$ is large enough and/or $a_2$ is small
enough.
\end{remark}

\begin{remark}
It is not hard to show that there holds the strong two-scale
resolvent convergence $A_\varepsilon \stackrel{2}\rightarrow A_0$,
see Definition \ref{defrc}. Namely, considering the resolvent
equation
\[
A_\varepsilon w^\varepsilon + w^\varepsilon = f^\varepsilon,
\]
where $f^\varepsilon \stackrel{2}\rightharpoonup f^0$, and employing
essentially the same arguments as above (cf. also \cite[Theorem
5.1]{Zhikov2}), one can pass to the limit as $\varepsilon \to 0$ in
the weak form of the resolvent equation choosing appropriate test
functions, cf. (\ref{7.38})--(\ref{75n}), to obtain that
$w^\varepsilon \stackrel{2}\rightharpoonup w^0$, with
\[
A_0 w^0 + w^0 = f^0.
\]
Further, arguing as in \cite[\S 4.3]{Zhikov2}, cf. also proof of
Lemma \ref{L7} above, one can show that the above weak two-scale
convergence implies the strong one, i.e. $w^\varepsilon
\stackrel{2}\to w^0$ as long as $f^\varepsilon \stackrel{2}\to f^0$,
which means the strong two-scale resolvent convergence by the
definition. The latter implies in particular the strong two-scale
convergence of spectral projectors ($P_\varepsilon(\lambda)
\stackrel{2}\to P_0(\lambda)$ if $\lambda$ is not an eigenvalue of
$A_0$), see \cite{ReedSimon, Zhikov1}, and has other nice
properties, however it does not imply in its own the convergence of
the spectra. The latter requires an additional (two-scale)
compactness property to hold, which Theorem \ref{th6} provides.
\end{remark}

\begin{remark}The function $v(x,y)$ could be represented as a product of
$u_0(x)\big|_{\Omega_1}$ and $\lambda _0 b(y)$, where $b(y)$ solves
(\ref{13a}) with $\lambda =\lambda _0$. Then $v(x,\varepsilon^{-1}
x)$  strongly two-scale converges to $v(x,y)$ by the mean value
property and the properties of two-scale convergence. Then
\begin{equation}
u^{\rm appr}(x,\varepsilon):=
    \left\{
      \begin{array}{ll}
         u_0(x)+v(x,x/\varepsilon), & x\in \Omega_0^\varepsilon,
          \\
          u_0(x), & x\in \mathbb{R}^n \backslash \Omega_0^\varepsilon,
      \end{array} \right. \label{91}
\end{equation}
also strongly two-scale converges to $u^0(x,y)$. Hence it
approximates the eigenfunction $u^\varepsilon(x)$:
\begin{equation}
\| u^{\rm appr}-u^\varepsilon\|^2_{L_2(\mathbb R^n)} \rightarrow 0.
\label{92}
\end{equation}
\end{remark}

Now, using the result of Theorem \ref{th6} we will discuss the
multiplicity properties of the eigenvalues $\lambda_{\varepsilon}$
and $\lambda_0$. Let us assume that the multiplicity of the
eigenvalue $\lambda_0$ of $A_0$ is $m$. Suppose that for a
subsequence $\varepsilon_k \to 0$ there exist $l$ (accounting for
multiplicities) eigenvalues of $A_\varepsilon $,
$\lambda_{\varepsilon_k,1}\leq \lambda_{\varepsilon_k,2}, \ldots
\leq \lambda_{\varepsilon_k,l}$, such that
$\lambda_{\varepsilon_k,i} \to \lambda_0$, $i=1,\ldots,l$. Let
$u_i^{\varepsilon_k}$ be the corresponding eigenfunctions
orthonormalized in $L^2(\mathbb{R}^n)$. It follows from Theorem
\ref{th6} that there exists a subsequence $k_m$ such that
\[
u_i^{\varepsilon_{k_m}} \stackrel{2}\rightarrow u_i^0, \,\,
i=1,\ldots,l,
\]
where $u_i^0$ are eigenfunctions of $A_0$ corresponding to
$\lambda_0$. In particular, due to the strong two-scale convergence,
we have convergence of the inner products as a consequence of the
convergence of norms:
\[
(u_i^{\varepsilon_{k_m}},
u_j^{\varepsilon_{k_m}})_{L^2(\mathbb{R}^n)} \rightarrow
(u_i^0,u_j^0)_{\mathcal {H}_0}.
\]
However $(u_i^{\varepsilon_{k_m}},
u_j^{\varepsilon_{k_m}})_{L^2(\mathbb{R}^n)} = \delta_{ij}$. Then
$u_i^0, \,\, i=1,\ldots,l$ are also orthonormal (in $\mathcal
{H}_0$), i.e. there exist at least $l$ linearly independent
eigenfunctions of $A_0$ corresponding to $\lambda_0$. Thus,  $l \leq
m$.

The results presented in \cite{SmysKam} remain also valid for the
setting of the problem in the present paper, i.e. when the
coefficients of the divergence form operator $A_\varepsilon$ are of
the form (\ref{4}). By Theorem 4.1 of \cite{SmysKam}, if $\lambda_0$
is an eigenvalue of the limit operator $A_0$ lying in a gap of its
essential spectrum, then for small enough $\varepsilon$, there exist
eigenvalues (or at least one eigenvalue) of $A_\varepsilon$ such
that
\[
|\lambda_{\varepsilon,i} - \lambda_0| \leq C \varepsilon^{1/2}, \,
i=1,\ldots,l(\varepsilon).
\]
Moreover, again by \cite[Thm 4.1]{SmysKam}, for any eigenfunction
$u^0_i$ of $A_0$ corresponding to $\lambda_0$ the related $u^{\rm
appr}_i$, see (\ref{91}), can be approximated by a linear
combination of the eigenfunctions of $A_\varepsilon$ corresponding
to $\lambda_{\varepsilon,i}, \, i=1,\ldots,l(\varepsilon)$. Since,
by the above, $(u^{\rm appr}_i, u^{\rm appr}_j)_{L^2(\mathbb{R}^n)}
\to \delta_{ij}$, as $\varepsilon \to 0$, $i,j = 1,\ldots ,m$, it is
not hard to show that $l(\varepsilon)\geq m$. Hence we conclude that
there exist exactly $m$ eigenvalues (counted with their
multiplicities) of $A_\varepsilon$ such that
\[
|\lambda_{\varepsilon,i} - \lambda_0| \leq C \varepsilon^{1/2}, \,
i=1,\ldots,m,
\]
where $m$ is a multiplicity of $\lambda_0$. In other words there is
an ``asymptotic one-to-one correspondence'' between isolated
eigenvalues and eigenfunctions of the operators $A_\varepsilon$ and
$A_0$.

\section{Identity of the essential spectra of $\widehat{A}_0$ and
$A_0$, convergence of the spectra of $A_\varepsilon$ in the sense of
Hausdorff}\label{s7}

By definition, the Hausdorff convergence of spectra,
$\sigma(A_\varepsilon) \stackrel{H}\rightarrow \sigma(A_0)$ as
$\varepsilon \to 0$, means that
\begin{itemize}
  \item for all $\lambda \in \sigma(A_0)$ there are $\lambda_\varepsilon \in
  \sigma(A_\varepsilon)$ such that $\lambda_\varepsilon \to
  \lambda$;
  \item if $\lambda_\varepsilon \in \sigma(A_\varepsilon)$ and $\lambda_\varepsilon \to
  \lambda$, then $\lambda \in \sigma(A_0)$.
\end{itemize}
We remind that $\widehat{A}_\varepsilon$ and $\widehat{A}_0$ denote
the `unperturbed' operators corresponding to $A_\varepsilon$ and
$A_0$, see Section \ref{s2}. It was shown in \cite{Zhikov1} that
$\sigma(\widehat{A}_\varepsilon) \stackrel{H}\rightarrow
\sigma(\widehat{A}_0)$ (the spectra of both
$\widehat{A}_\varepsilon$ and $\widehat{A}_0$ are purely essential).
In \cite{FigKlein} it is proved that the essential spectrum of a
divergence  form operator $\left. -\nabla \cdot a(x) \nabla \right.$
(where $a(x) \geq \delta > 0$ is a scalar function) remains
unperturbed with respect to the local perturbation of the
coefficient $a(x)$. Applying this assertion to the operator
$\widehat{A}_\varepsilon$ and its perturbation $A_\varepsilon$ we
conclude that
$\sigma(\widehat{A}_\varepsilon)=\sigma_{\mathrm{ess}}(A_\varepsilon)
\stackrel{H}\rightarrow \sigma(\widehat{A}_0)$. Let us assume that
$\sigma(\widehat{A}_0)=\sigma_{\mathrm{ess}}(A_0)$. Then
$\sigma_{\mathrm{ess}}(A_\varepsilon) \stackrel{H}\rightarrow
\sigma_{\mathrm{ess}}(A_0)$. In this case Theorem \ref{th6} together
with the results of \cite{SmysKam} imply the convergence of the
discrete spectra in the gaps ($\sigma_{\mathrm{disc}}(A_\varepsilon)
\stackrel{H}\rightarrow \sigma_{\mathrm{disc}}(A_0)$) and,
consequently, we would have $\sigma(A_\varepsilon)
\stackrel{H}\rightarrow \sigma(A_0)$. However, we cannot apply the
result of \cite{FigKlein} as it is stated to the case of the
two-scale operators $\widehat{A}_0$ and $A_0$. In this section we
prove the stability of the essential spectrum of $\widehat{A}_0$
with respect to the local perturbation of its coefficients,
establishing thereby the missing part of the reasoning. We do this
by direct means using the Weyl's criterium for the essential
spectrum of an operator, see e.g. \cite{Birman}.

\begin{theorem}\label{th7}
The essential spectra of the operators $\widehat{A}_0$ and $A_0$
coincide.
\end{theorem}
\begin{proof}

\textbf{Step 1.} First we describe the domains of $\widehat{A}_0$
and $A_0$. According to the Friedrichs extension procedure, see e.g.
\cite{ReedSimon}, a function $u$ belongs to $\mathcal D (A_0)$ if
and only if $u = u_0(x) + v(x,y) \in \mathcal V$ and there exists $h
\in \mathcal H_0$ such that
\[
B_0(u,w) = (h, w)_{\mathcal {H}_0}
\]
for all $w \in \mathcal V$, see (\ref{5.1})--(\ref{6}). If $u = u_0
+ v \in \mathcal D(A_0)$ then $u_0, v \in \mathcal D(A_0) $. Due to
the regularity properties of solutions of elliptic equations, $u_0
\in H^2_{\mathrm{loc}}$ everywhere away from the boundary of
$\Omega_2$.

Operator $\widehat{A}_0$ acting in the Hilbert space $\mathcal
{\widehat{H}}_0$ was described in \cite{Zhikov1} and is generated by
a (closed) symmetric and bounded from below bilinear form
$\widehat{B}_0(u,w)$ on a dense subspace $\mathcal {\widehat{V}}$ of
$\mathcal {\widehat{H}}_0$, where $\mathcal {\widehat{H}}_0$,
$\mathcal {\widehat{V}}$ and $\widehat{B}_0(u,w)$ are defined by
(\ref{5.1})--(\ref{6}) with $\Omega_2 = \emptyset$ and $\Omega_1 =
\mathbb{R}^n$. A function $u$ belongs to domain $\mathcal D
(\widehat{A}_0)$ if and only if $u = u_0(x) + v(x,y) \in \mathcal
{\widehat{V}}$ and there exists $h \in \mathcal {\widehat{H}}_0$
such that
\[
\widehat{B}_0(u,w) = (h, w)_{\mathcal {\widehat{H}}_0}
\]
for all $w \in \mathcal {\widehat{V}}$. If $u = u_0 + v \in \mathcal
D(\widehat{A}_0)$ then $u_0, v \in \mathcal D(\widehat{A}_0)$, $u_0
\in H^2(\mathbb{R}^n)$.

Let $A$ be a self-adjoint operator with domain $\mathcal D(A)$
acting in a Hilbert space $H$. By the Weyl's criterium, see e.g.
\cite{Birman}, condition $\lambda \in \sigma_{\mathrm{ess}}(A)$ is
equivalent to the existence of a singular sequence $u^{(k)}\in
\mathcal D (A)$, i.e. such that
\begin{equation}
0 < C_1 \leq \|u^{(k)}\|_H \leq C_2, \label{68}
\end{equation}
\begin{equation}
u^{(k)} \rightharpoonup 0 \mbox{ weakly in }  H, \label{69}
\end{equation}
\begin{equation}
(A - \lambda)u^{(k)} \rightarrow 0  \mbox{ strongly in }  H.
\label{70}
\end{equation}

\textbf{Step 2.} Let $\lambda \in
\sigma_{\mathrm{ess}}(\widehat{A}_0)$ and $u^{(k)} = u^{(k)}_0(x) +
v^{(k)}(x,y)$ be the corresponding singular sequence in $\mathcal
D(\widehat{A}_0) \subset \mathcal {\widehat{H}}_0$. We want to
construct on its basis a singular sequence for the operator $A_0$,
i.e. in $\mathcal D(A_0) \subset \mathcal H_0$ and satisfying
properties (\ref{68})--(\ref{70}). First notice that the gradient of
$u^{(k)}_0$ is bounded in $L^2(\mathbb{R}^n)$. Indeed, from
(\ref{6}) and (\ref{70}) we have
\begin{equation}
\|\nabla u^{(k)}_0\|^2_{L^2(\mathbb{R}^n)} \leq C
\widehat{B}_0(u^{(k)},u^{(k)}) = C \lambda
(u^{(k)},u^{(k)})_{\mathcal {\widehat{H}}_0} + o(1) \leq C.
\label{72}
\end{equation}

Let us define a cut-off function
\[
\eta_{k, R}(x) = \eta\left(\frac{1}{k}(|x|-R)\right),
\]
where $\eta \in C^2(\mathbb{R})$ is such that
\[
\eta(t) = \left\{
      \begin{array}{ll}
         1, & t \leq 0,
          \\
          0, & t \geq 1.
      \end{array} \right.
\]

Consider the following sequence, $u^{(k)}\eta_{k, R_k}\in \mathcal D
(\widehat{A}_0)$, where $R_k$ is chosen large enough so that $\|
u^{(k)}(1-\eta_{k, R_k})\|_{\mathcal {\widehat{H}}_0} \leq
\frac{1}{k}$. This sequence obviously satisfies (\ref{68}) regarding
the operator $\widehat{A}_0$.

Let us check property (\ref{70}). The operator $\widehat{A}_0$ acts
on a function $u\in H^2(\mathbb{R}^n) \subset \mathcal
{\widehat{H}}_0$ as follows\footnote[7]{If $u = u_0(x) + v(x,y)$
then $\widehat{A}_0 u = h \in \widehat{\mathcal{H}}_0$ implies $-
\nabla \cdot A^{\mathrm{hom}}\nabla u_0 = \langle h \rangle_y$ and
$- a_0 \Delta_y v = h(x,y), \, y \in Q_0$.}, cf. \cite{Zhikov1}. Let
\[
- \nabla \cdot A^{\mathrm{hom}}\nabla u(x) = f(x) \in
L^2(\mathbb{R}^n).
\]
Then, by the definition of $\widehat{A}_0$, we have
\[
\widehat{A}_0 u(x) = |Q_1|^{-1} \theta_{Q_1} (y) f(x) \in \mathcal
{\widehat{H}}_0.
\]
Note that
\[
\|\widehat{A}_0 u\|_{\mathcal {\widehat{H}}_0} = |Q_1|^{-1/2}
\|f\|_{L^2(\mathbb{R}^n)}.
\]
For $u^{(k)}\eta_{k, R_k}$ we derive
\[
\begin{aligned}
\widehat{A}_0 \left(u^{(k)}\eta_{k, R_k}\right) = \eta_{k, R_k}
\widehat{A}_0 u^{(k)} - |Q_1|^{-1} \theta_{Q_1} (y)\left( 2 \nabla
\eta_{k, R_k} \cdot A^{\mathrm{hom}}\nabla u^{(k)}_0 + u^{(k)}_0
\nabla \cdot A^{\mathrm{hom}}\nabla \eta_{k, R_k}\right).
\end{aligned}
\]
Thus we arrive at
\begin{equation}
\begin{aligned}
\left\|(\widehat{A}_0 - \lambda)(u^{(k)}\eta_{k,
R_k})\right\|_{\mathcal {\widehat{H}}_0} \leq \left\| \eta_{k, R_k}
(\widehat{A}_0 - \lambda)u^{(k)}\right\|_{\mathcal {\widehat{H}}_0}
+ |Q_1|^{-1/2}\left(2\,\left\|\nabla \eta_{k, R_k} \cdot
A^{\mathrm{hom}}\nabla u^{(k)}_0\right\|_{L^2(\mathbb{R}^n)} +
\right.
\\
\left. + \left\|u^{(k)}_0 \nabla \cdot A^{\mathrm{hom}}\nabla
\eta_{k, R_k}\right\|_{L^2(\mathbb{R}^n)}\right) =
 o(1) + \frac{1}{k} O\left(\left\|\nabla
u^{(k)}_0\right\|_{L^2(\mathbb{R}^n)}\right) + \frac{1}{k^2}
O\left(\left\| u^{(k)}_0\right\|_{L^2(\mathbb{R}^n)}\right).
\label{721}
\end{aligned}
\end{equation}
Due to (\ref{68}) and (\ref{72}) the latter converges to $0$ as
$k\to \infty$. Hence (\ref{70}) holds regarding $\widehat{A}_0$.

Now notice that if $\,\mathrm{supp}\, u \cap \overline{\Omega}_2 =
\emptyset$, then $u\in \mathcal D (\widehat{A}_0)$ if and only if
$u\in \mathcal D (A_0)$; besides $\widehat{A}_0 u= A_0 u$. We hence
next shift the supports of the elements of the sequence away from
$\Omega_2$ ensuring also that the new sequence is weakly convergent
to maintain (\ref{69}). Since $\mathrm{supp}\, \eta_{k, R_k}$ is a
closed ball of radius $R_k+k$ centered at the origin, the shift of
$x$ by $\xi_k := \left(R_k+2k+\mathrm{diam} (\Omega_2) \right)\xi$
for every $k$, where $\xi$ is an arbitrary unit vector from
$\mathbb{R}^n$, will do the job. Hence, for the given $\lambda$ we
have constructed a singular sequence
\[
w^{(k)}(x,y) = u^{(k)}(x+\xi_k, y) \, \eta_{k, R_k} (x+\xi_k),
\]
satisfying all the properties (\ref{68})--(\ref{70}) for the
operator $A_0$. Namely, the translational invariance of
$\widehat{A}_0$ in $x$ ensures that (\ref{68}) and (\ref{70}) are
satisfied. Finally, (\ref{69}) follows from the pointwise
convergence of $w^{(k)}$ to zero as $k \to \infty$ (since for any
fixed $x$, $w^{(k)}(x,y) = 0$ for large enough $k$). Thus $\lambda
\in \sigma_{\mathrm{ess}}(A_0)$.

\textbf{Step 3.} Suppose now that $\lambda \in
\sigma_{\mathrm{ess}}(A_0)$ and $u^{(k)} = u^{(k)}_0(x) +
v^{(k)}(x,y)$ is the corresponding singular sequence. Let $R$ be
such that $\overline{\Omega}_2 \subset B_R$. There are only two
alternative possibilities\footnote[8]{Let $A_{ki}:=\|u^{(k)}(1 -
\theta_{B_{R+i}})\|_{\mathcal H_0}$ and let $\delta _i :=
\sup\limits_k A_{ki}$. Then either $\delta _i \to 0$ giving
(\ref{73}) or $\delta _i \nrightarrow 0$ yielding (\ref{74}).}:
\begin{itemize}
  \item There exists a sequence $\delta_i \to 0$ such that for any $i
  \in \mathbb{N}$
\begin{equation}
\|u^{(k)}(1 - \theta_{B_{R+i}})\|_{\mathcal H_0} \leq \delta_i
\label{73}
\end{equation}
for all $k$.
  \item There exist a constant $M>0$ and subsequences $k(j)\to \infty, \, i(j) \to \infty$ as $j\to \infty$ such that
\begin{equation}
\|u^{(k(j))}(1 - \theta_{B_{R+i(j)}})\|_{\mathcal H_0} \geq M
\label{74}
\end{equation}
for all $j$.
\end{itemize}

Let (\ref{73}) take place. The sequence $\nabla u^{(k)}_0$ is
bounded in $L^2(\mathbb{R}^n)$, cf. (\ref{72}). From (\ref{73}) and
\begin{equation}
\|f\|_{L^2(\mathbb{R}^n)} = \|f\|_{\mathcal H_0}, \mbox{ for all }
f\in L^2(\mathbb{R}^n)\subset \mathcal H_0, \label{75}
\end{equation}
it follows that
\[
u^{(k)}_0 \rightarrow  u(x) \mbox{ in }  L^2(\mathbb{R}^n),
\]
up to a subsequence. The reasoning leading to this assertion is
essentially identical to the one in (\ref{59})--(\ref{7.9}) and is
not reproduced. From (\ref{69}) and the latter we conclude that
\[
v^{(k)}(x,y) \rightharpoonup - \, u(x) \mbox{ weakly in }  \mathcal
{H}_0.\] Hence, on one hand, we have
\[
\left(u, v^{(k)}\right)_{\mathcal {H}_0} \rightarrow - \left(u, u
\right)_{\mathcal {H}_0} = - \int\limits_{\mathbb{R}^n} u^2 \, dx.
\]
On the other hand,
\[
\left(u, v^{(k)}\right)_{\mathcal {H}_0} =
\int\limits_{\mathbb{R}^n} \int\limits_{Q_0}u \, v^{(k)} \,dy \,dx =
\left(u \, \theta_{Q_0}(y) , v^{(k)}\right)_{\mathcal {H}_0}
\rightarrow  - \left(u \, \theta_{Q_0}(y), u \right)_{\mathcal
{H}_0} = - \, |Q_0| \int\limits_{\mathbb{R}^n} u^2 \, dx.
\]
Comparing the last two formulas, conclude that at $u \equiv 0$, i.e.
\begin{equation}
u^{(k)}_0 \rightarrow  0 \mbox{ in } L^2(\mathbb{R}^n). \label{76}
\end{equation}
Denote $A_0 u^{(k)}$ by $g^{(k)}(x,y) = g_0^{(k)}(x) + h^{(k)}(x,y)
\in \mathcal {H}_0$. From (\ref{70}) and (\ref{75}) we get the
following convergence:
\[
\|g_0^{(k)} - \lambda u^{(k)}_0 \|_{L^2(\mathbb{R}^n)} \to 0,
\]
\begin{equation}
\|h^{(k)} - \lambda v^{(k)} \|_{\mathcal H_0} \to 0. \label{82}
\end{equation}
Then from (\ref{76}) we have
\begin{equation}
g^{(k)}_0 \rightarrow  0 \mbox{ in } L^2(\mathbb{R}^n). \label{83}
\end{equation}

Analogously to \cite{Zhikov2} we define a self-adjoint operator
$A_y$ acting in $L^2(\Omega_1 \times Q_0)$ by
\[
A_y v = - a_0 \Delta_y v = p, \quad p \in L^2(\Omega_1 \times Q_0).
\]
The domain of the operator, $\mathcal D(A_y) \subset L^2(\Omega_1,
H^1_0(Q_0))$, is the set of all the solution of this equation.
Similarly can be defined operator $\widehat{A}_y$ acting in
$L^2(\mathbb{R}^n \times Q_0)$. One can easily check the following
properties: $\mathcal D(A_y) \subset \mathcal D(A_0)$, $\mathcal
D(\widehat{A}_y) \subset \mathcal D(\widehat{A}_0)$, $\sigma (A_y) =
\sigma (\widehat{A}_y)$, $\sigma (A_y) \subset \sigma_{\mathrm{ess}}
(A_0)$ and, in particular,
\begin{equation}
\sigma (\widehat{A}_y) \subset \sigma_{\mathrm{ess}}
(\widehat{A}_0). \label{84}
\end{equation}
It is not difficult to see (by analyzing (\ref{6}), see also
\cite{Zhikov2}) that
\begin{equation}
A_y v^{(k)} = g_0^{(k)} \theta_{\Omega_1}(x) \theta_{Q_0}(y) +
h^{(k)}. \label{85}
\end{equation}
(Note that $A_y v^{(k)} \neq A_0 v^{(k)}$.)

Combining (\ref{82}), (\ref{83}) and (\ref{85}) we arrive at
\[
\|(A_y - \lambda) v^{(k)} \|_{L^2(\Omega_1 \times Q_0)} =
\|g_0^{(k)} \theta_{Q_0}(y) + h^{(k)} - \lambda v^{(k)}
\|_{L^2(\Omega_1 \times Q_0)} \to 0.
\]
This implies that $\lambda$ belongs to the spectrum of $A_y$ (notice
that (\ref{68}) holds for $v^{(k)}$ via (\ref{76})). Hence $\lambda
\in \sigma_{\mathrm{ess}}(\widehat{A}_0)$, see (\ref{84}).

Now let (\ref{74}) hold. Consider a sequence $w^{(j)} =
u^{(k(j))}(1-\eta_{i(j), R}) \in \mathcal D(\widehat{A}_0)$ (we
remind that $R$ is large enough to ensure $\Omega_2 \subset\subset
B_R$). Then
\[
\|w^{(j)}\|_{\mathcal {\widehat{H}}_0} \geq
\|u^{(k(j))}(1-\theta_{B_{R+i(j)}})\|_{\mathcal H_0} \geq M,
\]
i.e. (\ref{68}) is satisfied for $w^{(j)}$. Since the sequence
$1-\eta_{i(j), R}$ tends to $0$ pointwise, (\ref{69}) is valid.
Analogously to (\ref{721}) we derive
\begin{equation}
\begin{aligned}
&\|(\widehat{A}_0 - \lambda)w^{(j)}\|_{\mathcal {\widehat{H}}_0} =
\|({A}_0 - \lambda)w^{(j)}\|_{\mathcal {{H}}_0} \to 0,
\end{aligned}
\end{equation}
yielding (\ref{70}). Thus, we conclude that $\lambda \in
\sigma_{\mathrm{ess}}(\widehat{A}_0)$, completing the proof of the
theorem.
\end{proof}

\begin{remark}
Theorem \ref{th7} combined with \cite{Zhikov1} implies that
$\sigma_{\mathrm{ess}}(A_0) = \left\{\lambda : \beta (\lambda)\geq 0
\right\}\cup \sigma(A_y)$. Using the methods of \cite{Zhikov1} it is
not hard to show further that $\sigma_{\mathrm{ess}}(A_0)$ contains
no point spectrum (in particular, no embedded eigenvalues) except if
$\lambda$ is an eigenvalue of $A_y$ corresponding to an
eigenfunction with zero mean. It is natural to conjecture (cf.
\cite{Zhikov1}) that, outside these eigenvalues, the spectrum is
absolutely continuous and the ``eigenfunctions of the continuous
spectrum'' are $u(x,y,\lambda) = u_0(x,\lambda)(1 + \lambda
b(y,\lambda))$, where $u_0(x,\lambda)$ are solutions of the
appropriate scattering problems:
\begin{equation}\label{92a}
\begin{aligned}
\nabla \cdot A^{\mathrm{hom}} \nabla u_0 + \beta (\lambda) u_0 =0,&
\,\, x\in \mathbb{R}^n\backslash \overline{\Omega}_2,
\\
a_2 \Delta u_0 + \lambda u_0 =0,& \,\, x\in \Omega_2
\end{aligned}
\end{equation}
with the appropriate matching condition at $\partial \Omega_2$ and
radiation condition at infinity. A detailed study of this as well as
of the convergence of the related generalized eigenfunctions (cf.
\cite{Zhikov1} for the defect-free case) is beyond the scope of the
present paper.
\end{remark}

Summarizing the main results of the present paper we conclude that
Theorems \ref{th6} and \ref{th7} together with the results of
\cite{FigKlein,SmysKam} (see the discussions at the end of Section
\ref{s6} and in the beginning of the present section) establish the
validity of Theorem \ref{th1}.

\label{lastpage}


\begin{thebibliography}{11}

\bibitem{Agmon}
Agmon S., 1982, Lectures on exponential decay of solutions of
second-order elliptic equations: bounds on eigenfunctions of N-body
Schredinger operators. {\itshape Math. Notes,} {\bfseries 29}.
Princeton University Press.

\bibitem{Al}
Allaire G., 1992, Homogenization and two-scale convergence.
{\itshape SIAM J. Math. Anal.,} {\bfseries 23}, 1482–-1518.

\bibitem{ADH}
Arbogast T., Douglas J. Jr. and Hornung U., 1990, Derivation of the
double porosity model of single phase flow via homogenization
theory. {\itshape SIAM J. Math. Anal.,} {\bfseries 21}, no. 4,
823–-836.

\bibitem{Babych}
Babych N. A., Kamotski I. V., Smyshlyaev V. P., 2006, Homogenization
in periodic media with doubly high contrasts. {\itshape Networks and
Heterogeneous Media,} in press. Available online at:
http://www.newton.cam.ac.uk/preprints/NI07063.pdf.

\bibitem{BCH}
Barbaroux J. M., Combes J. M., and Hislop P. D., 1997, Localization
near band edges for random Schr\"{o}dinger operators. {\itshape
Helv. Phys. Acta,} {\bfseries 70}, 16--43

\bibitem{Bellieud}
Bellieud M., 2005,  Homogenization of evolution problems for a
composite medium with very small and heavy inclusions. {\itshape
ESAIM Control Optim. Calc. Var.,} {\bfseries 11}, no. 2, 266--284.

\bibitem{Birman}
Birman M.S., Solomyak M.Z., 1987, Spectral Theory of Self-Adjoint
Operators in Hilbert Space. {\itshape D. Reidel Publishing Company.}

\bibitem{Bouchitte}
Bouchitt\'{e}, G., Felbacq, D., 2004, Homogenization near resonances
and artificial magnetism from dielectrics. {\itshape C. R. Math.
Acad. Sci. Paris,} {\bfseries 339}, no. 5, 377--382.

\bibitem{BMP}
Bourgeat A., Mikeli´c A. and Piatnitski A., 2003, On the double
porosity model of a single phase flow in random media. {\itshape
Asymptot. Anal.,} {\bfseries 34}, no. 3--4, 311–-332.

\bibitem{Briane}
Briane M., 2003,  Homogenization of the Stokes equations with
high-contrast viscosity. {\itshape J. Math. Pures Appl.,} {\bfseries
82}, no. 7, 843--876.

\bibitem{CSZH}
Cherednichenko K.D., Smyshlyaev V.P. and Zhikov V.V., 2006,
Non-local homogenised limits for composite media with highly
anisotropic periodic fibres. {\itshape Proc. Roy. Soc. Edinb. A,}
{\bfseries 136}, no. 1, 87–-114.

\bibitem{CH}
Cherednichenko K.D., 2006, Two-scale asymptotics for non-local
effects in composites with highly anisotropic fibres. {\itshape
Asymptot. Anal.,} {\bfseries 49}, no. 1-2, 39–-59.

\bibitem{FigKlein}
Figotin A., Klein A., 1997, Localized classical waves created by
defects. {\itshape J. Statist. Phys.,} {\bfseries 86}, no. 1-2,
165–-177.

\bibitem{Hempell}
Hempell R., Lienau J., 2000, Spectral properties of periodic media
in the large coupling limit. {\itshape Comm. Partial Diff.
Equations,} {\bfseries 25}, 1445–-1470.

\bibitem{ZhiKozOle}
Jikov V.V., Kozlov S.M., Oleinik O.A., 1994, Homogenization of
differential operators and integral functionals. Springer-Verlag,
Berlin.

\bibitem{SmysKam}
Kamotski I.V., Smyshlyaev V.P., 2006, Localized modes due to defects
in high contrast periodic media via homogenization. BICS preprint
3/06. Available online at:
www.bath.ac.uk/math-sci/preprints/BICS06\_3.pdf.

\bibitem{Kuchment}
Kuchment P., 2001, The mathematics of photonic crystals, in
Mathematical Modeling in Optical Science. {\itshape Frontiers in
Applied Mathematics, SIAM, Philadelphia,} {\bfseries 22}, 207--272.

\bibitem{Mikh}
Mikhailov V.P., 1978, Partial differential equations. Mir, Moscow.
Translated from Russian.

\bibitem{Ngu}
Nguetseng G., 1989, A general convergence result for a functional
related to the theory of homogenization. {\itshape SIAM J. Math.
Anal.,} {\bfseries 20}, 608–-623.

\bibitem{ReedSimon}
Reed M., Simon B., 1978, Methods of modern mathematical physics.
Academic Press.

\bibitem{Zhikov2}
Zhikov V.V., 2000, On an extension of the method of two-scale
convergence and its applications. {\itshape (Russian) Mat. Sb.,}
{\bfseries 191}, no. 7, 31--72; {\itshape translation in Sb. Math.,}
{\bfseries 191}, no. 7-8, 973--1014.

\bibitem{Zhikov1}
Zhikov V.V., 2004, Gaps in the spectrum of some elliptic operators
in divergent form with periodic coefficients. {\itshape(Russian)
Algebra i Analiz,} {\bfseries 16}, no. 5, 34--58; 2005, {\itshape
translation in St. Petersburg Math. J.,}  {\bfseries 16}, no. 5,
773--790.

\end{thebibliography}
\end{document}